\documentclass[10pt,reqno]{amsart}

\newtheorem{theo}{Theorem}[part]

\usepackage{amsmath,amssymb,amsthm,graphicx,url}

\usepackage[bookmarksnumbered, colorlinks, plainpages]{hyperref}
\hypersetup{colorlinks=true,linkcolor=blue, anchorcolor=green, citecolor=violet, urlcolor=cyan, filecolor=magenta, pdftoolbar=true}

\usepackage[english]{babel}
\usepackage[T1]{fontenc}
\usepackage{color}
\usepackage{standalone}
\usepackage{enumitem}
\usepackage{tikz-cd}
\usepackage{fbox}
\usepackage{mdframed}
\usepackage[shortcuts]{extdash}

\usepackage[a4paper, left=2.8cm, right=2.8cm, top=3.5cm, bottom=3.5cm]{geometry}

\newtheorem{maintheo}{Theorem}

\newtheorem{lemm}[theo]{Lemma}

\newtheorem{prop}[theo]{Proposition}
\newtheorem{coro}[theo]{Corollary}

\theoremstyle{definition}
\newtheorem{defi}[theo]{Definition}

\theoremstyle{remark}
\newtheorem*{rem}{Remark} 

\numberwithin{equation}{section}

\def\B{{\mathbb B}}
\def\C{{\mathbb C}}

\def\R{{\mathbb R}}
\def\S{{\mathbb S}}
\def\T{{\mathbb T}}

\def\pp{\leq}
\def\pg{\geq}
\def\bra{\langle}
\def\ket{\rangle}

\newcommand{\fonct}[5]{#1\,:\left\{\begin{array}{ccc} #2&\rightarrow&#3 \\ #4&\mapsto&#5 \end{array}\right.}
\newcommand{\fonctbis}[4]{\left\{\begin{array}{ccc} #1&\rightarrow&#2 \\ #3&\mapsto&#4 \end{array}\right.}

\DeclareMathOperator{\ima}{ran}

\DeclareMathOperator{\re}{Re}

\DeclareMathAlphabet\cat{OMS}{cmsy}{b}{n}

\usepackage{newtxtext,newtxmath}

\makeatletter

\providecommand{\proofname}{Preuve}
\makeatother

\begin{document}

\title{A note on Pisier's method in interpolation of abstract Hardy spaces}

\author{Hugues Moyart}
\address{UNICAEN, CNRS, LMNO, 14000 Caen, France}
\email{hugues.moyart@unicaen.fr}


\begin{abstract}
In his approach to Jones’ theorem on the interpolation of Hardy spaces on the torus, Pisier introduced an original method allowing the computation of complex interpolation spaces by means of real interpolation techniques. This approach has been successfully extended to noncommutative analytic Hardy spaces arising from subdiagonal algebras. In this paper, we formulate and prove an abstract version of Pisier’s method in a more general setting. The method is then applied in the study of noncommutative martingale transforms. 
\end{abstract} 

\maketitle

\setcounter{tocdepth}{3}
\tableofcontents

\clearpage

\part*{Introduction}

Interpolation of Hardy spaces has been extensively studied since the seminal work of Peter W. Jones. We briefly recall the results obtained in this context. Let $\T$ be the unit circle. For $1\pp p\pp\infty$, let $H_p(\T)$ denote the associated Hardy space on $\T$, i.e. the subspace of $L_p(\T)$ of functions whose negative Fourier coefficients vanish. Jones established in \cite{JonesHardy2} that there is a universal positive constant $C$ such that for every $f\in H_1(\T)+H_\infty(\T)$, and $t>0$, we have
\[K_t(f,H_1(\T),H_\infty(\T))\pp CK_t(f,L_1(\T),L_\infty(\T))\]
where $K$ refers to the $K$-functional used in the construction of Lions-Peetre real interpolation spaces. According to the terminology introduced by Pisier in \cite{PisierHardy}, one can reformulate Jones' theorem by saying that the subcouple $(H_1(\T),H_\infty(\T))$ is $K$-closed in the compatible couple $(L_1(\T),L_\infty(\T))$. This implies that for every $0<\theta<1$, we have
\begin{equation}\label{introduction_eq1}
(H_1(\T),H_\infty(\T))_{\theta,p_\theta,K}=H_{p_\theta}(\T)
\end{equation}
with equivalent norms, with constants depending only on $\theta$, where $1/p_\theta=1-\theta$. Here the left-hand side refers to the Lions-Peetre real interpolation method. Jones also proved in \cite{JonesHardy1} the analogue of \eqref{introduction_eq1} for the complex method, which reads as follows: if $0<\theta<1$, then
\begin{equation}\label{introduction_eq2}
    [H_1(\T),H_\infty(\T)]_{\theta}=H_{p_\theta}(\T)
\end{equation}
with equivalent norms, with constants depending only on $\theta$, where $1/p_\theta=1-\theta$. In his approach to Jones' results, Pisier introduced an original method in \cite{PisierHardy} allowing him to deduce \eqref{introduction_eq2} from \eqref{introduction_eq1}. One of the features of Pisier’s method is its natural compatibility with noncommutative frameworks. As a result, Jones' results have been successfully extended to noncommutative analytic Hardy spaces associated with subdiagonal algebras: first by Pisier in \cite{PisierHardy} for the case of upper triangular matrices, and later by Bekjan in \cite{BekjanSubdiagonalHardySpaces} in the general setting.

\vspace{10pt}

The main contribution of this paper is an abstract formulation of Pisier’s argument under minimal assumptions. This level of abstraction not only clarifies the original argument but also allows one to treat various interpolation problems within a unified framework. The need for an abstract formulation is motivated by the recent $K$-closedness results obtained in \cite{Moyart2024} in the context of Hardy spaces of noncommutative martingales. To better explain our approach, we now introduce the mathematical setting of this paper. All the unexplained notions and notations are recalled in \textcolor{red}{Part 1}.

\vspace{10pt}

If $M$ is a semifinite von Neumann algebra, and $P$ is any projection on $L_2(M)$ and $1\pp p\pp\infty$, we define $H_p(P)$ to be the closure of $L_p(M)\cap P(L_2(M))$ in $L_p(M)$ (for the norm topology if $p<\infty$, and the w*-topology if $p=\infty$).

\begin{mdframed}[backgroundcolor=black!10,rightline=false,leftline=false,topline=false,bottomline=false,skipabove=10pt]
Let $M$ be a semifinite von Neumann algebra. A projection $P$ on $L_2(M)$ is said to be of \textit{Jones-type} if it satisfies the technical assumptions (defined in \textcolor{red}{Part 2}) and if the subcouples $(H_1(P),H_\infty(P))$ and $(H_1(P^{\bot}),H_\infty(P^{\bot}))$ are both $K$-closed in the compatible couple $(L_1(M),L_\infty(M))$ with a universal constant.  
\end{mdframed}
\vspace{5pt}

Actually, in the body of this paper we use a slightly different set of assumptions, which are easier to verify in practice. In \textcolor{red}{Part 2} we prove the equivalence between the two formulations.

\begin{mdframed}[backgroundcolor=black!10,rightline=false,leftline=false,topline=false,bottomline=false,skipabove=10pt]
Let $M$ be a semifinite von Neumann algebra. A projection $P$ on $L_2(M)$ is said to satisfy \textit{Pisier's method assumptions} if, for every semifinite von Neumann algebra $N$, the amplified projection $Q:=P\otimes I$ on the Hilbertian tensor product $L_2(M)\otimes L_2(N)=L_2(M\bar{\otimes}N)$ is of Jones-type, and the algebraic tensor product $H_\infty(P^{\bot})\odot L_\infty(N)$ is a subspace of $H_\infty(Q^{\bot})$.
\end{mdframed}
\vspace{5pt}

The advantage of our formulation is that it remains entirely within the setting of normed spaces, thereby avoiding technical difficulties that arise when dealing with $L_p$-spaces for $0<p<1$. This contrasts with Pisier's original approach. The price to pay for this simplification is that we cannot expect to obtain results for the compatible couple $(H_1(P),H_\infty(P))$. Nevertheless, by following Pisier's original arguments, we establish in \textcolor{red}{Part 2} the following result. 

\begin{mdframed}[skipabove=20pt]
\textbf{\upshape Theorem A} \text{\upshape (Pisier's method).} \textit{Let $M$ be a semifinite von Neumann algebra, and let $P$ be a projection on $L_2(M)$ that satifies Pisier's method assumptions. If $1<p<\infty$ and $0<\theta<1$ then
\[[H_p(P),H_1(P)]_\theta=H_{p_\theta}(P)\]
with equivalent norms, with constants depending only on $p,\theta$, where 
\[\frac{1}{p_\theta}=\frac{1-\theta}{p}+\theta.\]}
\end{mdframed}
\vspace{15pt}

In \textcolor{red}{Part 3}, we apply our extension of Pisier's method to the study of noncommutative martingales. As a consequence, we obtain new results on martingale projections that extend the study initiated in \cite{Moyart2024}, which we now briefly describe.

\vspace{10pt}

Let $M$ be a semifinite von Neumann algebra equipped with a filtration $(M_n)_{n\pg1}$ with associated conditional expectations denoted $(E_n)_{n\pg1}$ and associated increment projections denoted $(D_n)_{n\pg1}$ defined by 
\[D_n:=E_n-E_{n-1},\ \ \ \ \ \text{for}\ n\pg1.\]
(with the convention $E_0:=0$). Let $I$ be a fixed set of positive integers, and let $P$ be the \textit{martingale projection} associated with the data $(M,(M_n)_{n\pg1},I)$, i.e. the projection on $L_2(M)$ such that, if $x\in L_2(M)$ then
\[P(x)=\sum_{n\in I}D_n(x),\ \ \ \text{in}\ L_2(M).\]
Then, if $1\pp p\pp\infty$ we have
\[H_p(P)=\big\{x\in L_p(M)\ :\ \forall n\notin I,\ D_n(x)=0\big\}.\]
In \cite{Moyart2024} it is essentially proved that the martingale projection $P$ is of Jones-type. By applying our extented version of Pisier's method, we then obtain the following new result.

\begin{mdframed}[skipabove=10pt]
\textbf{\upshape Theorem.} \text{\upshape } \textit{Let $M$ be a semifinite von Neumann algebra equipped with a filtration. If $1<p<\infty$ and $0<\theta<1$ then
\[[H_p(P),H_1(P)]_\theta=H_{p_\theta}(P)\]
with equivalent norms, with constants depending only on $p,\theta$, where 
\[\frac{1}{p_\theta}=\frac{1-\theta}{p}+\theta.\]}
\end{mdframed}
\vspace{10pt}

By using well-known transference techniques, the above result implies several additional statements in the context of complex interpolation of Hardy spaces of noncommutative martingales that are detailed in \textcolor{red}{Part 4}. In particular, we recover all the classical results obtained by Musat in \cite{MusatHardy} about interpolation noncommutative Hardy spaces.

\newpage

\part{Preliminaries}

\section{Abstract interpolation theory}

The material in this section is taken from \cite{JansonRealInterpolation} and \cite{BerghInterpolation}.

\subsection{Compatible couples}

A \textit{compatible couple} is a couple $(E_0,E_1)$ of subspaces of a common Hausdorff topological vector space $E$, such that $E_j$ is equipped with a complete norm that makes the inclusion $E_j\to E$ continuous, for $j\in\{0,1\}$. Then the intersection space $E_0\cap E_1$ and the sum space $E_0+E_1$ are canonically equipped with the complete norms $\|\cdot\|_{E_0\cap E_1}$ and $\|\cdot\|_{E_0+E_1}$ given by the expressions
\[\|u\|_{E_0\cap E_1}:=\max\big\{\|u\|_{E_0},\ \|u\|_{E_1}\big\},\]
\[\|u\|_{E_0+E_1}:=\inf\big\{\|u_0\|_{E_0}+\|u_1\|_{E_1}\ :\  u=u_0+u_1,\ u_0\in E_0,u_1\in E_1\big\}.\]
An \textit{intermediate space} for a compatible couple $(E_0,E_1)$ is a subspace $E_\theta$ of $E_0+E_1$ that contains $E_0\cap E_1$, and that is equipped with a complete norm that makes the inclusions $E_0\cap E_1\to E_\theta$ and $E_\theta\to E_0+E_1$ both continuous. If $E_{\theta_0}$, $E_{\theta_1}$ are intermediate spaces for a compatible couple $(E_0,E_1)$, then their sum $E_{\theta_0}+E_{\theta_1}$ and their intersection $E_{\theta_0}\cap E_{\theta_1}$ are also intermediate spaces for $(E_0,E_1)$ when equipped with the corresponding sum norm $\|\cdot\|_{E_{\theta_0}+E_{\theta_0}}$ and intersection norm $\|\cdot\|_{E_{\theta_0}\cap E_{\theta_1}}$ as defined above.

\subsection{Compatible bounded operators}

Let $(E_0,E_1)$ and $(F_0,F_1)$ be two compatible couples. A \textit{compatible bounded operator} $(E_0,E_1)\to(F_0,F_1)$ is an operator $T:E_0+E_1\to F_0+F_1$ such that, if $j\in\{0,1\}$, then $T$ maps $E_j$ into $F_j$, and $T:E_j\to F_j$ is bounded. In this situation, we set
\[\|T\|_{(E_0,E_1)\to(F_0,F_1)}:=\max\big\{\|T\|_{E_0\to F_0},\|T\|_{E_1\to F_1}\big\}.\]
Let $T:(E_0,E_1)\to(F_0,F_1)$ be a compatible bounded operator. Note that $T$ is injective resp. surjective, bijective if and only if $T:E_j\to F_j$ is, for $j\in\{0,1\}$. 

We say that $T$ is an \textit{embedding/quotient} of compatible couples if $T:E_j\to F_j$ is an embedding/quotient of normed spaces for $j\in\{0,1\}$ (recall that a bounded operator $T:E\to F$ between normed spaces is an embedding/quotient if it is injective/surjective and the induced bounded operator $E/\ker T\to\ima T$ is an isomorphism of normed spaces). We say that $T$ is an \textit{isomorphism} of compatible couples if $T:E_j\to F_j$ is an isomorphism of normed spaces, for $j\in\{0,1\}$. 

We say that $T$ is \textit{contractive} if $\|T\|_{(E_0,E_1)\to(E_0,E_1)}\pp1$. We say that $T$ is an \textit{isometric embedding/coisometric quotient} of compatible couples if $T:E_j\to F_j$ is an isometric embedding/coisometric quotient of normed spaces for $j\in\{0,1\}$ (recall that a quotient of normed spaces $T:E\to F$ is coisometric if the induced isomorphism of normed spaces $E/\ker T\to F$ is isometric). We say that $T$ is an \textit{isometric isomorphism} of compatible couples if $T:E_j\to F_j$ is an isometric isomorphism of normed spaces, for $j\in\{0,1\}$. 

\begin{rem}
There is an obvious way to define the category of compatible couples and compatible (contractive) bounded operators. The isomorphisms in this category correspond to the (isometric) isomorphisms of compatible couples.
\end{rem}

An \textit{interpolation space} with constant $C$ for a compatible couple $(E_0,E_1)$ is an intermediate space $E_\theta$ for $(E_0,E_1)$, such that, if $T:(E_0,E_1)\to(E_0,E_1)$ is a compatible bounded operator, then $T$ maps $E_\theta$ into $E_\theta$ and $\|T\|_{E_\theta\to E_\theta}\pp C\|T\|_{(E_0,E_1)\to(E_0,E_1)}$. An \textit{exact interpolation space} is an interpolation space with constant $C\pp1$. The sum/intersection of (exact) interpolation spaces is again an (exact) interpolation space. More generally, an \textit{interpolation pair} with constant $C$ for a pair of compatible couples $(E_0,E_1)$ and $(F_0,F_1)$ is a pair of intermediate spaces $E_\theta$ and $F_\theta$ for $(E_0,E_1)$ and $(F_0,F_1)$ respectively, such that, if $T:(E_0,E_1)\to(F_0,F_1)$ is a compatible bounded operator, then $T$ maps $E_\theta$ into $F_\theta$ and $\|T\|_{E_\theta\to F_\theta}\pp C\|T\|_{(E_0,E_1)\to(F_0,F_1)}$.

An \textit{exact interpolation pair} is an interpolation pair with constant $C\pp1$. 

\subsection{Interpolation functors}

An \textit{interpolation functor} with constant $C$ is a map $\mathcal{F}$ that assigns to each compatible couple $(E_0,E_1)$ an intermediate space $\mathcal{F}(E_0,E_1)$, such that, if $(E_0,E_1)$, and $(F_0,F_1)$ is a pair of compatible couples, then $\mathcal{F}(E_0,E_1)$ and $\mathcal{F}(F_0,F_1)$ is an interpolation pair with constant $C$ for $(E_0,E_1)$ and $(F_0,F_1)$ (in this situation, if $(E_0,E_1)$ is a compatible couple, then $\mathcal{F}(E_0,E_1)$ is necessarily an interpolation space with constant $C$ for $(E_0,E_1)$).

An \textit{exact interpolation functor} is an interpolation functor with constant $C\pp1$. 

\begin{rem}
For instance, the map $\Sigma$ (resp. $\Delta$) that assigns to each compatible couple $(E_0,E_1)$ the sum space $E_0+E_1$ (resp. the intersection space $E_0\cap E_1$) is an exact interpolation functor.
\end{rem}

If $\mathcal{F}$ is an (exact) interpolation functor, then $\mathcal{F}$ defines in an obvious way a functor from the category of compatible couples and compatible (contractive) bounded operators to the category of complete normed spaces and (contractive) bounded operators. 

\subsection{Subcouples} 

A \textit{subcouple} of a compatible couple $(E_0,E_1)$ is a couple $(A_0,A_1)$ where $A_j$ is a closed subspace of $E_j$ for $j\in\{0,1\}$. In this situation, the couple $(A_0,A_1)$ inherits an unique structure of compatible couple that turns the inclusion $A_0+A_1\to E_0+E_1$ into an isometric embedding of compatible couples $(A_0,A_1)\to(E_0,E_1)$. Thus, if $\mathcal{F}$ is an (exact) interpolation functor, then the inclusion
\[\fonctbis{\mathcal{F}(A_0,A_1)}{\mathcal{F}(E_0,E_1)}{u}{u}\]
is a well-defined bounded (contractive) injective operator, but it it may fail to be an embedding of normed spaces (this is in particular the case for $\mathcal{F}=\Sigma$, but note that, however, in the case $\mathcal{F}=\Delta$, the inclusion $A_0\cap A_1\to E_0\cap E_0$ is an isometric embedding of normed spaces).

\subsection{Quotient couples}

Let $(A_0,A_1)$ be a subcouple of a compatible couple $(E_0,E_1)$, and let 
\[\fonctbis{E_0+E_1}{(E_0+E_1)/(A_0+A_1)}{u}{\overline{u}}\]
denote the projection. If $E_\theta$ is an intermediate space for $(E_0,E_1)$ such that the subspace 
\[A_\theta:=E_\theta\cap(A_0+A_1)=\big\{u\in E_\theta\ :\ \overline{u}=0\big\}\]
is closed in $E_\theta$, then
\[E_\theta/A_\theta:=\big\{v\in (E_0+E_1)/(A_0+A_1)\ :\ \exists u\in E_\theta,\ v=\overline{u}\big\}\]
is a subspace of $(E_0+E_1)/(A_0+A_1)$ and is equipped with the complete norm $\|\cdot\|_{E_\theta/A_\theta}$ given by the expression
\[\|v\|_{E_\theta/A_\theta}=\inf_{u\in E_\theta,\ \overline{u}=v}\|u\|_{E_\theta}.\]
Moreover, it is clear that the inclusion $E_\theta/A_\theta\to(E_0+E_1)/(A_0+A_1)$ is continuous (where $(E_0+E_1)/(A_0+A_1)$ is equipped with the quotient topology, which may not be Hausdorff as $A_0+A_1$ may fail to be closed in $E_0+E_1$), and there is a canonical isometric isomorphism from $E_\theta/A_\theta$ to the usual quotient space of $E_\theta$ by $A_\theta$.

A subcouple $(A_0,A_1)$ of a compatible couple $(E_0,E_1)$ is \textit{normal} if $A_0+A_1$ is closed in $E_0+E_1$ and if $A_j=(A_0+A_1)\cap E_j$ for $j\in\{0,1\}$. In this situation, by the above discussion the couple of normed spaces $(E_0/A_0,E_1/A_1)$ is well-defined and inherits an unique structure of compatible couple that turns the projection $E_0+E_1\to(E_0+E_1)/(A_0+A_1)$ into a coisometric quotient of compatible couples $(E_0,E_1)\to(E_0/A_0,E_1/A_1)$. Thus, if $\mathcal{F}$ is an (exact) interpolation functor, then the projection 
\[\fonctbis{\mathcal{F}(E_0,E_1)}{\mathcal{F}(E_0/A_0,E_1/A_1)}{u}{\overline{u}}\]
is a well-defined bounded (contractive) operator, but it may fail to be a quotient of normed spaces, or to be merely surjective (note that, however, in the case $\mathcal{F}=\Sigma$, the projection $E_0+E_1\to(E_0/A_0)+(E_1/A_1)$ is a coisometric quotient of normed spaces, so that $(E_0/A_0)+(E_1/A_1)=(E_0+E_1)/(A_0+A_1)$ with equal norms). 

\begin{prop}\label{preliminaries_quotient_couple}
Let $(A_0,A_1)$ be a normal subcouple of a compatible couple $(E_0,E_1)$. Then the projection $E_0\cap E_1\to(E_0/A_0)\cap(E_1/A_1)$ is surjective.
\end{prop}
\begin{proof}
Let $u\in(E_0/A_0)\cap(E_1/A_1)$. Then $u=\overline{u}_0=\overline{u}_1$ with $u_0\in E_0$, $u_1\in E_1$. As $\overline{u_0-u_1}=0$, we can write $u_0-u_1=a_0-a_1$ with $a_0\in A_0$, $a_1\in A_1$. If $v:=u_0-a_0=u_1-a_1$, then $v\in E_0\cap E_1$ and $\overline{v}=u$, as desired.
\end{proof}

\subsection{Complementation} 

A subcouple $(A_0,A_1)$ of a compatible couple $(E_0,E_1)$ is \textit{complemented} with constant $C$ if there is a bounded (contractive) compatible operator $P:(E_0,E_1)\to(E_0,E_1)$ such that $P:E_j\to E_j$ is idempotent with range $A_j$, and $\|P\|_{E_j\to E_j}\pp C$, for $j\in\{0,1\}$.

A subcouple $(A_0,A_1)$ of a compatible couple $(E_0,E_1)$ is \textit{contractively complemented} if it is complemented with constant $C\pp1$.

If $(A_0,A_1)$ is a (contractively) complemented subcouple of a compatible couple $(E_0,E_1)$, then for every (exact) interpolation functor $\mathcal{F}$, the inclusion $\mathcal{F}(A_0,A_1)\to\mathcal{F}(E_0,E_1)$ is an (isometric) embedding of normed spaces, with range $(A_0+A_1)\cap\mathcal{F}(E_0,E_1)$, and the projection
$\mathcal{F}(E_0,E_1)\to\mathcal{F}(E_0/A_0,E_1/A_1)$ is a (coisometric) quotient of normed spaces, with kernel $\mathcal{F}(A_0,A_1)$.

\subsection{Duality}

Let $(E_0,E_1)$ be a compatible couple. If $E_\theta$ is an intermediate space for $(E_0,E_1)$ such that $E_0\cap E_1$ is dense in $E_\theta$, then 
\[E_\theta^*:=\big\{ u^*\in(E_0\cap E_1)^*,\ \sup_{u\in E_0\cap E_1,\ \|u\|_{E_\theta}\pp1}|\bra u,u^*\ket|<\infty\big\}\]
is a subspace of $(E_0\cap E_1)^*$ and is equipped with the complete norm $\|\cdot\|_{E_\theta^*}$ given by the expression
\[\|u^*\|_{E_\theta^*}=\sup_{u\in E_0\cap E_1,\ \|u\|_{E_\theta}\pp1}|\bra u,u^*\ket|.\]
Moreover, it is clear that the inclusion $E_\theta^*\to(E_0\cap E_1)^*$ is continuous, and there is a canonical isometric isomorphism from $E_\theta^*$ to the usual dual space of $E_\theta$.  

A compatible couple $(E_0,E_1)$ is \textit{regular} if $E_0\cap E_1$ is dense in $E_j$ for $j\in\{0,1\}$. In this situation and by the above discussion, the couple of complete normed spaces $(E_0^*,E_1^*)$ is well-defined and inherits a canonical structure of compatible couple. Moreover, we have $E_0^*+E_1^*=(E_0\cap E_1)^*$ with equal norms and $E_0^*\cap E_1^*=(E_0+E_1)^*$ with equal norms. 

\begin{prop}[\cite{JansonRealInterpolation}(Proposition 6.2)]
Let $(A_0,A_1)$ be a regular subcouple of a regular compatible couple $(E_0,E_1)$ such that $A_0+A_1$ is closed in $E_0+E_1$. For $j\in\{0,1\}$, let consider the orthogonal
\[A_j^{\bot}:=\big\{\phi\in E_j^*\ :\ \phi(u)=0,\ \forall u\in A_j\big\}\]
Then $(A_0^{\bot},A_1^{\bot})$ is a normal subcouple of the dual couple $(E_0^*,E_1^*)$, and there is a canonical isomorphism of compatible couples between $(E_0^*/A_0^{\bot},E_1^*/A_1^{\bot})$ and $(A_0^*,A_1^*)$.
\end{prop}

\section{The complex method}

Let $\S:=\{z\in\C\ :\ 0<\re z<1\}$ denote the open unit strip in the complex plane, with boundary $\partial\S=\{it\ :\ t\in\R\}\cup\{1+it\ :\ t\in\R\}$ and closure $\overline{\S}:=\{z\in\C\ :\ 0\pp\re z\pp1\}$. 

Let $(E_0,E_1)$ be a compatible couple. Let $\mathcal{F}(E_0,E_1)$ denote the space of functions $f:\overline{\S}\to E_0+E_1$ continuous on $\overline{\S}$, holomorphic on $\S$, such that $f(j+it)\in E_j$ for $t\in\R$, $j\in\{0,1\}$, and such that the functions $\R\to E_j$, $t\mapsto f(j+it)$ are continuous and vanishes at infinity, for $j\in\{0,1\}$. If $f\in\mathcal{F}(E_0,F_1)$, we set
\[\|f\|_{\mathcal{F}(E_0,E_1)}:=\max_{j\in\{0,1\}}\sup_{t\in\R}\|f(j+it)\|_{E_j}.\]
If $0<\theta<1$, the \textit{complex interpolation space} $[E_0,E_1]_\theta$ is the subspace of $E_0+E_1$ of elements of the form $f(\theta)$ with $f\in\mathcal{F}(E_0,E_1)$. It is equipped with the complete norm $\|\cdot\|_{[E_0,E_1]_\theta}$ given by the expression
\[\|u\|_{[E_0,E_1]_\theta}:=\inf\big\{\|f\|_{\mathcal{F}(E_0,E_1)}\ :\ f\in\mathcal{F}(E_0,E_1),\ u=f(\theta)\big\}.\]
This construction yields, for fixed $0<\theta<1$, an exact interpolation functor. By convention, we set $(E_0,E_1)_{0}:=E_0$ and $(E_0,E_1)_{1}:=E_1$.

\begin{prop}[\cite{BerghInterpolation}(Theorem 4.2.2)]
Let $(E_0,E_1)$ be a compatible couple. Then $E_0\cap E_1$ is dense in $[E_0,E_1]_\theta$ for every $0<\theta<1$.
\end{prop}

The following duality theorem is folklore. It directly follows from \cite{Wolff}[Lemma 2].

\begin{theo}[Duality Theorem]\label{ComplexInterpolation_Duality}
Let $(E_0,E_1)$ be a regular compatible couple and $0<\theta<1$. Then for every $u\in E_0\cap E_1$, we have 
\[\|u\|_{[E_0,E_1]_\theta}=\sup_{\protect\substack{u^*\in[E_0^*,E_1^*]_{\theta}\\\|u\|_{(E_0^*,E_1^*)_\theta}\pp1}}|\bra u,u^*\ket|.\]
\end{theo}

\begin{coro}\label{ComplexInterpolation_Dual}
Let $(E_0,E_1)$ be a regular compatible couple, and let $E_\theta$ be an intermediate space for $(E_0,E_1)$ such that $E_0\cap E_1$ is dense in $E_\theta$. Let $0<\theta<1$ and let $C$ be a constant, such that, if $u^*\in E_0^*\cap E_1^*$, then $u^*\in E_\theta^*$ with
\[\|u^*\|_{E_\theta^*}\pp C\|u^*\|_{[E_0^*,E_1^*]_\theta}.\]
Then, for every $u\in E_\theta$, we have $u\in[E_0,E_1]_\theta$, with
\[\|u\|_{[E_0,E_1]_\theta}\pp C\|u\|_{E_\theta}.\]
\end{coro}
\begin{proof}
As $E_0^*\cap E_1^*$ is dense in $[E_0^*,E_1^*]_\theta$, from the hypothesis we deduce that for every $u^*\in[E_0^*,E_1^*]_\theta$, we have $u^*\in E_\theta^*$, with $\|u^*\|_{E_\theta^*}\pp C\|u^*\|_{[E_0^*,E_1^*]_\theta}$. Thus, if $u\in E_0\cap E_1$, then by Theorem \ref{ComplexInterpolation_Duality} we get
\begin{align*}
\|u\|_{[E_0,E_1]_\theta}&\pp\sup_{\protect\substack{u^*\in [E_0^*,E_1^*]_\theta\\\|u^*\|_{[E_0^*,E_1^*]_\theta}\pp1}}|\bra u,u^*\ket|\pp \sup_{\protect\substack{u^*\in [E_0^*,E_1^*]_\theta\\\|u^*\|_{[E_0^*,E_1^*]_\theta}\pp1}}\|u\|_{E_\theta}\|u^*\|_{E_\theta^*}\\
&\pp C\sup_{\protect\substack{u^*\in [E_0^*,E_1^*]_\theta\\\|u^*\|_{[E_0^*,E_1^*]_\theta}\pp1}}\|u\|_{E_\theta}\|u^*\|_{[E_0^*,E_1^*]_{\theta}}\pp C\|u\|_{E_\theta}.
\end{align*}
As $E_0\cap E_1$ is dense in $E_\theta$, the desired conclusion follows.
\end{proof}

The reiteration theorem for the complex method is classical. A proof can be found in \cite{BerghInterpolation}[Theorem 4.6.1]

\begin{theo}[Reiteration theorem]\label{ComplexInterpolation_Reiteration}
Let $(E_0,E_{\theta_0},E_{\theta_1},E_1)$ be a compatible quadruplet such that $E_0\cap E_1$ is dense in $E_{\theta_0}\cap E_{\theta_1}$ and assume that
\[E_{\theta_0}:=[E_0,E_1]_{\theta_0}\ \ \ \ \ \text{and}\ \ \ \ \ E_{\theta_1}:=[E_0,E_1]_{\theta_1},\]
where $0\pp\theta_0<\theta_1\pp1$. Let $0<\eta<1$. Then we have
\[[E_{\theta_0},E_{\theta_1}]_{\eta}=[E_0,E_1]_{\theta_\eta}\]
with equal norms, where $\theta_\eta:=(1-\eta)\theta_0+\eta\theta_1$.
\end{theo}



\section{The real method}

\subsection{\texorpdfstring{$K$}{}-functionals} 

Let $(E_0,E_1)$ be a compatible couple. The \textit{$K$-functional} of $u\in E_0+E_1$ is defined for $t>0$ as
\[K_t(u)=K_t(u,E_0,E_1):=\inf\big\{\|u_0\|_{E_0}+t\|u_1\|_{E_1}\ :\ u_0\in E_0,\ u_1\in E_1,\ u=u_0+u_1\big\}.\]
For fixed $t>0$, $K_t$ is an equivalent norm on $E_0+E_1$. If $(E_0,E_1)$ and $(F_0,F_1)$ are two compatible couples and $T:(E_0,E_1)\to(F_0,F_1)$ a compatible bounded operator, then \[K_t(Tu,F_0,F_1)\pp\|T\|_{(E_0,E_1)\to(F_0,F_1)}K_t(u,E_0,E_1)\]
for every $u\in E_0+E_1$ and $t>0$. 

A \textit{$K$\-/method parameter} is a complete normed space $\Phi$ of (class of) Lebesgue measurable functions with variable $t>0$ such that,
\begin{enumerate}[nosep]
    \item[$\triangleright$] if $f,g\in\Phi$ with $|g|\pp|f|$ then $\|g\|_{\Phi}\pp\|f\|_{\Phi}$, 
    \item[$\triangleright$] the function $t\mapsto1\wedge t$ belongs to $\Phi$.
\end{enumerate}
If $\Phi$ is a $K$\-/method parameter and $(E_0,E_1)$ is a compatible couple, then the \textit{$K$-method interpolation space}
\[K_\Phi(E_0,E_1):=\big\{u\in E_0+E_1\ :\ t\mapsto K_t(u,E_0,E_1)\in\Phi\big\}\]
is a subspace of $E_0+E_1$ and is equipped with the complete norm $\|\cdot\|_{K_\Phi(E_0,E_1)}$ given by the expression
\[\|u\|_{K_\Phi(E_0,E_1)}:=\|t\mapsto K_t(u,E_0,E_1)\|_{\Phi}.\]
This construction defines an exact interpolation functor $K_\Phi$ called the \textit{$K$\-/method} with parameter $\Phi$. 

If $(A_0,A_1)$ is a subcouple of a compatible couple $(E_0,E_1)$, then we have
\[K_t(u,E_0,E_1)\pp K_t(u,A_0,A_1),\ \ \ \ \text{for}\ u\in A_0+A_1,\ t>0\]
A subcouple $(A_0,A_1)$ of a compatible couple $(E_0,E_1)$ is \textit{$K$\-/closed} with constant $C$ if 
\[K_t(u,A_0,A_1)\pp CK_t(u,E_0,E_1)\]
for every $u\in A_0+A_1$ and $t>0$. 

\begin{rem}
If $(A_0,A_1)$ is a $K$-closed subcouple of a compatible couple $(E_0,E_1)$, then $A_0+A_1$ is closed in $E_0+E_1$ (in other words, the inclusion $A_0+A_1\to E_0+E_1$ is an embedding of normed spaces).
\end{rem}

The following proposition follows directly from the definitions.

\begin{prop}\label{preliminaries_Kclosedness}
If $(A_0,A_1)$ is a $K$-closed subcouple of a compatible couple $(E_0,E_1)$, then for every $K$\-/method parameter $\Phi$, the inclusion $K_\Phi(A_0,A_1)\to K_\Phi(E_0,E_1)$ is an embedding of normed spaces, with range $(A_0+A_1)\cap K_\Phi(E_0,E_1)$.
\end{prop}

The following key result is due to Pisier. A detailed proof can be found in \cite{JansonRealInterpolation}[Theorem 6.1].

\begin{theo}[Pisier]\label{preliminaries_Pisier}
Let $(A_0,A_1)$ be a $K$-closed subcouple of a regular compatible couple $(E_0,E_1)$ with constant $C$. Then the subcouple $(A_0^{\bot},A_1^{\bot})$ is $K$-closed in $(E_0^*,E_1^*)$ with a constant depending only on $C$.
\end{theo}

The following result, which is a consequence of the generalized Holmstedt formula for $K$-functionals, is taken from \cite{JansonRealInterpolation}[Theorem 2.2] 

\begin{theo}\label{preliminaries_Holmstedt}
Let $(A_0,A_1)$ be a $K$-closed subcouple of a compatible couple $(E_0,E_1)$ with constant $C$. Let $\Phi_0,\Phi_1$ be $K$-parameter spaces. Let denote
\[E_{\theta_0}:=K_{\Phi_0}(E_0,E_1)\ \ \ \ \ \text{and}\ \ \ \ \ E_{\theta_1}:=K_{\Phi_1}(E_0,E_1).\]
and we adopt analogous notations for the compatible couple $(A_0,A_1)$. The following assertions hold.
\begin{enumerate}
    \item The subcouple $(A_{\theta_0},A_{\theta_1})$ is $K$-closed in $(E_{\theta_0},E_{\theta_1})$ with a constant depending only on $C$, $\Phi_0$, $\Phi_1$.
    \item If $\Phi_0$ contains a continuous unbounded positive concave function, then the subcouple $(A_{\theta_0},A_1)$ is $K$-closed in $(E_{\theta_0},E_1)$ with a constant depending only on $C$, $\Phi_0$.
    \item If $\Phi_1$ contains a continuous unbounded positive concave function, then the subcouple $(A_0,A_{\Phi_1})$ is $K$-closed in $(E_0,E_{\Phi_1})$ with a constant depending only on $C$, $\Phi_1$.
\end{enumerate}
\end{theo}

\subsection{\texorpdfstring{$J$}{}-functionals} 

Let $(E_0,E_1)$ be a compatible couple. The \textit{$J$-functional} of $u\in E_0\cap E_1$ is defined for $t>0$ as
\[J_t(u)=J_t(u,E_0,E_1):=\|u\|_{E_0\cap tE_1}=\max\big\{\|u\|_{E_0},t\|u\|_{E_1}\big\}.\]
For fixed $t>0$, $J_t$ is an equivalent norm on $E_0\cap E_1$. If $T:(E_0,E_1)\to(F_0,F_1)$ is a bounded compatible operator between compatible couples, then \[J_t(Tu,F_0,F_1)\pp\|T\|_{(E_0,E_1)\to(F_0,F_1)}J_t(u,E_0,E_1)\]
for every $u\in E_0\cap E_1$ and $t>0$. 

A \textit{$J$\-/method parameter} is a complete normed space $\Phi$ of (class of) Lebesgue measurable functions with variable $t>0$ such that
\begin{enumerate}[nosep]
    \item[$\triangleright$] if $f,g\in\Phi$ with $|g|\pp|f|$ then $\|g\|_{\Phi}\pp\|f\|_{\Phi}$,
    \item[$\triangleright$] the function $t\mapsto1\wedge t$ belongs to $\Phi$,
    \item[$\triangleright$] if $f\in\Phi$, then $\int_{0}^{\infty}|f(t)|(1\wedge t^{-1})\frac{dt}{t}\pp c_\Phi\|f\|_{\Phi}$ where $c_\Phi$ depends on $\Phi$ only.
\end{enumerate}
(every $J$\-/method parameter is in particular a $K$\-/method parameter). If $\Phi$ is a $J$\-/method parameter and $(E_0,E_1)$ is a compatible couple, if we denote 
\[L_1^\Phi(dt/t,E_0,E_1):=\Big\{f\in L_1(dt/t,E_0+E_1)\ :\ t\mapsto J_t(f(t))\in\Phi\Big\}\]
then the \textit{$J$-method interpolation space}
\[J_\Phi(E_0,E_1):=\Big\{u\in E_0+E_1\ :\exists f\in L_1^\Phi(dt/t,E_0,E_1),\ u=\int_{0}^{\infty}f(t)\frac{dt}{t}\Big\}\]
is a subspace of $E_0+E_1$ and is equipped with the complete norm $\|\cdot\|_{J_\Phi(E_0,E_1)}$ given by the expression
\[\|u\|_{J_\Phi(E_0,E_1)}:=\inf\Big\{\|t\mapsto J_t(f(t))\|_{\Phi}\ :\ f\in L_1^\Phi(dt/t,E_0,E_1),\ u=\int_{0}^{\infty}f(t)\frac{dt}{t}\Big\}.\]
This construction defines an exact interpolation functor $J_\Phi$ called the \textit{$J$\-/method} with parameter $\Phi$. 

A subcouple $(A_0,A_1)$ of a compatible couple $(E_0,E_1)$ is \textit{$J$\-/closed} with constant $C$ if it is normal in $(E_0,E_1)$ and if for every $u\in E_0\cap E_1$ and $t>0$, there is $v\in E_0\cap E_1$ such that $u-v\in A_0+A_1$ and
\[J_t(v,E_0,E_1)\pp CJ_t(\overline{u},E_0/A_0,E_1/A_1).\]
Again, the following proposition follows directly from the definitions.

\begin{prop}[\cite{JansonRealInterpolation}(Theorem 4.1)]
Let $(A_0,A_1)$ be a normal subcouple of a compatible couple $(E_0,E_1)$ and let $C$ be a constant. Then $(A_0,A_1)$ is $K$-closed in $(E_0,E_1)$ with a constant depending only on $C$ if and only if $(A_0,A_1)$ is $J$-closed in $(E_0,E_1)$ with a constant depending only on $C$.
\end{prop}

\begin{prop}\label{preliminaries_Jclosedness}
If $(A_0,A_1)$ is a $J$-closed subcouple of a compatible couple $(E_0,E_1)$, then for every $J$\-/method parameter $\Phi$, the projection $J_\Phi(E_0,E_1)\to J_\Phi(E_0/A_0,E_1/A_1)$ is an quotient of normed spaces, with kernel $J_\Phi(A_0,A_1)$.
\end{prop}

\subsection{The real interpolation spaces}

Let $0<\theta<1$ and $1\pp p\pp\infty$. Let $\Phi_{\theta,p}$ denote the space of Lebesgue\-/measurable functions $f$ with variable $t>0$ such that 
\[\|f\|_{\Phi_{\theta,p}}:=\|t\mapsto t^{-\theta}f(t)\|_{L_p(dt/t)}<\infty\]
Then $\Phi_{\theta,p}$ is a $J$\-/method parameter (and thus also a $K$\-/method parameter). If $(E_0,E_1)$ is a compatible couple, let us denote \[(E_0,E_1)_{\theta,p,K}:=K_{\Phi_{\theta,p}}(E_0,E_1),\ \ \ \ \\ \ (E_0,E_1)_{\theta,p,J}:=J_{\Phi_{\theta,p}}(E_0,E_1).\]
By convention, we set $(E_0,E_1)_{0,p,K}=(E_0,E_1)_{0,p,J}:=E_0$ and $(E_0,E_1)_{1,p,K}=(E_0,E_1)_{1,p,J}:=E_1$ for every $1\pp p\pp\infty$. They are the so called \textit{real interpolation spaces}.

The following results are taken from \cite{BerghInterpolation}.

\begin{theo}[Equivalence theorem]\label{RealInterpolation_Equivalence}
Let $(E_0,E_1)$ be a compatible couple. Then, for $0<\theta<1$ and $1\pp p\pp\infty$, we have
\[(E_0,E_1)_{\theta,p,K}=(E_0,E_1)_{\theta,p,J}\]
with equivalent norms, with constants depending only on $\theta$.
\end{theo}

\begin{theo}[Reiteration theorem]
Let $(E_0,E_1)$ be a compatible couple. We set
\[E_{\theta_0}:=(E_0,E_1)_{\theta_0,p_0,K}\ \ \ \ \ \text{and}\ \ \ \ \ E_{\theta_1}:=(E_0,E_1)_{\theta_1,p_1,K},\]
where $0\pp\theta_0<\theta_1\pp1$ and $1\pp p_0,p_1\pp\infty$ (with the convention $E_{\theta_0}=E_0$ if $\theta_0=0$ and $E_{\theta_1}=E_1$ if $\theta_1=1$). Let $0<\eta<1$ and $1\pp p\pp\infty$. Then,
\[(E_{\theta_0},E_{\theta_1})_{\eta,p,K}=(E_0,E_1)_{\theta_\eta,p,K}\]
with equivalent norms, with constants depending only on $\theta_0,\theta_1,p_0,p_1,\eta,p$, where $\theta_\eta:=(1-\eta)\theta_0+\eta\theta_1$.
\end{theo}

\subsection{Quasi-complementation}

The notion of $K$-closedness used in \cite{KislyakovXu} is slightly different from the one we use here. In order to avoid any confusion we introduce a new terminology.

\begin{defi}
A subcouple $(A_0,A_1)$ of a compatible couple $(E_0,E_1)$ is \textit{quasi\-/complemented} with constant $C$ if, for every $a\in A_0+A_1$, if $a=u_0+u_1$ with $u_j\in E_j$ for $j\in\{0,1\}$, then $a=a_0+a_1$ with $a_j\in A_j$ and $\|a_j\|_{E_j}\pp C\|u_j\|_{E_j}$ for $j\in\{0,1\}$.
\end{defi}

Let $(A_0,A_1)$ be a subcouple of a compatible couple $(E_0,E_1)$. If $(A_0,A_1)$ is complemented in $(E_0,E_1)$ with constant $C$, then $(A_0,A_1)$ is clearly quasi-complemented in $(E_0,E_1)$ with constant $C$. Moreover, it is clear that if $(A_0,A_1)$ is quasi-complemented in $(E_0,E_1)$ with constant $C$, then it is $K$-closed in $(E_0,E_1)$ with constant $C$, and in addition we have $A_j=E_j\cap(E_0+E_1)$ for $j\in\{0,1\}$, so that it is also normal (and thus $J$-closed) in $(E_0,E_1)$. 

\begin{prop}
Let $(A_0,A_1)$ be a $K$-closed subcouple of a compatible couple $(E_0,E_1)$ such that $A_j=E_j\cap(E_0+E_1)$ for $j\in\{0,1\}$. Then $(A_0,A_1)$ is quasi-complemented in $(E_0,E_1)$ with constant $4C$.
\end{prop}
\begin{proof}
Let $a\in A_0+A_1$, and write $a=u_0+u_1$ with $u_j\in E_j$ for $j\in\{0,1\}$. If $u_j=0$, then $a\in(A_0+A_1)\cap E_j=A_j$ and there is nothing to prove. Thus, we can assume $u_0,u_1\neq0$. Set $t:=\|u_0\|_{E_0}/\|u_1\|_{E_1}$. Then, we can write $a=a_0+a_1$ with $a_j\in A_j$ and 
\[\|a_0\|_{E_0}+t\|a_1\|_{E_1}\pp 2K_t(a,A_0,A_1)\pp 2C K_t(a,E_0,E_1)\pp 2C(\|u_0\|_{E_0}+t\|u_1\|_{E_1})=4C\|u_0\|_{E_0}.\]
This gives the estimates
\[\|a_0\|_{E_0}\pp 4C\|u_0\|_{E_0},\ \ \ \ \|a_1\|_{E_1}\pp 4C\|u_1\|_{E_1}.\]
\end{proof}

The following statement is adapted from \cite{KislyakovXu}[Theorem 2].

\begin{theo}[Kislyakov/Xu]\label{preliminaries_KislyakovXu}
 Let $(A_0,A_{\theta_0},A_{\theta_1},A_1)$ be a subquadruple of a compatible quadruple
 
 $(E_0,E_{\theta_0},E_{\theta_1},E_1)$, such that
\[E_{\theta_0}=(E_0,E_{\theta_1})_{\eta_0,p_0,J},\ \ \ \ E_{\theta_1}=(E_{\theta_0},E_1)_{\eta_1,p_1,J}\]
for given $0<\eta_0,\eta_1<1$ and $1\pp p_0,p_1\pp\infty$. In addition, we make the following two assumptions.
\begin{enumerate}[nosep]
    \item the subcouple $(A_0,A_{\theta_1})$ is quasi\-/complemented in $(E_0,E_{\theta_1})$ and there is a continuous inclusion $A_{\theta_0}\to (A_0,A_{\theta_1})_{\eta_0,p_0,K}$, both with constant $C_0$.
    \item the subcouple $(A_{\theta_0},A_1)$ is quasi\-/complemented in $(E_{\theta_0},E_1)$ and there is a continuous inclusion $A_{\theta_1}\to (A_{\theta_0},A_1)_{\eta_1,p_1,K}$, both with constant $C_1$.
\end{enumerate}
Then the subcouple $(A_0,A_1)$ is quasi\-/complemented in $(E_0,E_1)$ with a constant depending only on $C_0$, $C_1$, $\eta_0$, $\eta_1$.
\end{theo}

\section{Noncommutative integration theory}

The material in this section is taken from \cite{DoddsSukochevIntegration}.

\subsection{Semifinite von Neumann algebras}

Let $M$ be a \textit{(semi)finite von Neumann algebra}, i.e. a von Neumann algebra equipped with a normal (semi)finite faithful (n.s.f.) trace $\tau$. Let $H$ denote the Hilbert space on which $M$ acts. A closed and densely defined operator $x$ on $H$ with polar decomposition $x=u|x|$ and spectral decomposition $|x|=\int_{0}^{\infty}sde_s$ is \textit{affiliated} with $M$ if $u\in M$ and $e_s\in M$ for all $s>0$. The \textit{distribution function} of $x$ is the right-continuous decreasing function of the variable $s>0$ denoted $\lambda_x$ such that
\[\lambda_x(s)=\tau(1-e_s),\ \ \ \text{for}\ s>0.\]
The \textit{singular function} of $x$ is the right-continuous decreasing function of the variable $s>0$ denoted $\mu_x$ such that
\[\mu_x(s):=\inf\big\{t>0\ :\ \lambda_x(t)\pp s\big\},\ \ \text{for}\ s>0.\]
A closed and densely defined operator $x$ on $H$ is \textit{$\tau$\-/measurable} if it is affiliated with $M$ and if its distribution function (or its singular function) takes at least one finite value. Any element of $M$ is $\tau$\-/measurable. The set $L_0(M)$ of $\tau$\-/measurable operators then admits a canonical structure of $\ast$\-/algebra, so that the inclusion $M\to L_0(M)$ is a $\ast$-morphism and $\tau$ is canonically extended to the positive part $L_0(M)_+$ of $L_0(M)$ so that
\[\tau(x)=\int_{0}^{\infty}\lambda_x(s)ds=\int_{0}^{\infty}\mu_x(s)ds,\ \ \ \text{for}\ x\in L_0(M)_+.\]
We set 
\[L_\infty(M):=M\]
and 
\[L_1(M):=\big\{x\in L_0(M)\ :\ \tau(|x|)<\infty\big\}\]
Then $L_1(M)$ is a subspace of $L_0(M)$ and is equipped with the complete norm $\|\cdot\|_{L_1(M)}$ given by the expression
\[\|x\|_{L_1(M)}=\tau(|x|).\]
Moreover, the trace $\tau$ extends to a positive and contractive linear form on $L_1(M)$. 

\subsection{The compatible couple \texorpdfstring{$(L_1,L_\infty)$}{}}

Let $M$ be a semifinite von Neumann algebra, with trace denoted $\tau$. The $\ast$-algebra $L_0(M)$ admits a canonical Hausdorff topology so that the inclusions $L_1(M)\to L_0(M)$ and $L_\infty(M)\to L_0(M)$ become continuous, turning the couple $(L_1(M),L_\infty(M))$ into a compatible couple. 

\begin{prop}
Let $E(M)$ be an exact interpolation space for $(L_1(M),L_\infty(M))$. If $x\in E(M)$, $a,b\in M$, then $axb\in E(M)$ with $\|axb\|_{E(M)}\pp\|a\|_M\|x\|_{E(M)}\|b\|_M$.
\end{prop}

\begin{theo}[\cite{DoddsSukochevIntegration}(Theorem 3.9.16)]\label{integration_Holmstedt}
Let $x\in L_0(M)$. Then $x\in L_1(M)+L_\infty(M)$ if and only if
for every $t>0$, we have
\[\int_{0}^{t}\mu_x(s)ds<\infty\]
and in that case, we have
\[K_t(x,L_1(M),L_\infty(M))=\int_{0}^{t}\mu_x(s)ds,\ \ \ \ \text{for}\ t>0.\]
\end{theo}

\begin{theo}[\cite{DoddsInterpolationLp}(Theorem 2.4)]
Let $E(M)$ be an exact interpolation space for $(L_1(M),L_\infty(M))$. For every $x,y\in L_0(M)$ with $K_t(x)\pp K_t(y)$ for $t>0$, if we have $y\in E(M)$, then $x\in E(M)$ with $\|x\|_{E(M)}\pp\|y\|_{E(M)}$. 
\end{theo}

\begin{rem}
As a consequence, if $x\in E(M)$, then we have $x^*,|x|\in E(M)$ with $\||x|\|_{E(M)}=\|x^*\|_{E(M)}=\|x\|_{E(M)}$.
\end{rem}

\vspace{5pt}

If $E(M)$ is an exact interpolation space for $(L_1(M),L_\infty(M))$, its \textit{K\"othe dual} 
\[E^{\times}(M):=\big\{y\in L_0(M)\ :\ \forall x\in E(M),\ xy\in L_1(M)\big\}\]
\[=\big\{y\in L_0(M)\ :\ \forall x\in E(M),\ yx\in L_1(M)\big\}\]
is a subspace of $L_0(M)$ and is equipped with the complete norm $\|\cdot\|_{E^{\times}(M)}$ given by the expression
\[\|y\|_{E^{\times}(M)}=\sup_{x\in E(M),\ \|x\|_{E(M)}\pp1}|\tau(xy)|=\sup_{x\in E(M),\ \|x\|_{E(M)}\pp1}|\tau(yx)|,\ \ \ \ \ \ \text{for}\ y\in E^{\times}(M).\]
 By definition, if $x\in E(M)$ and $y\in E^{\times}(M)$, then $xy,yx\in L_1(M)$ and 
\[\max\{\|xy\|_{L_1(M)},\|yx\|_{L_1(M)}\}\pp\|x\|_{E(M)}\|y\|_{E^{\times}(M)}.\ \ \ \ \ \ \ \ \ \textit{(generalized H\"older's inequality)}\]
If $E(M)$, $F(M)$ are two exact interpolation spaces for $(L_1(M),L_\infty(M))$, then we have 
\[(E\cap F)^{\times}(M)=E^{\times}(M)+F^{\times}(M),\ \ \ \ \ (E+F)^{\times}(M)=E^{\times}(M)\cap F^{\times}(M)\]
with equal norms.

\begin{theo}[\cite{DoddsSukochevIntegration}(Proposition 3.4.8, Theorem 3.4.24)]
We have $L_1(M)^{\times}=L_\infty(M)$ and $L_\infty(M)^{\times}=L_1(M)$ with equal norms.
\end{theo}

\begin{theo}[\cite{DoddsSukochevIntegration}(Corollary 4.3.9)]
Let $E(M)$ be an exact interpolation space for $(L_1(M),L_\infty(M))$. Then the K\"othe dual $E^{\times}(M)$ is an exact interpolation space for $(L_1(M),L_\infty(M))$. Moreover, the sequilinear form
\[\fonctbis{E(M)\times E^{\times}(M)}{\C}{(x,y)}{\tau(xy^*)}\]
is non-degenerate, and thus, defines a canonical duality between $E(M)$ and $E^{\times}(M)$, called \textbf{trace duality}. Finally, $L_1(M)\cap L_\infty(M)$ separates both the points of $E(M)$ and $E^{\times}(M)$ w.r.t. this duality.
\end{theo}


Let $E(M)$ be an exact interpolation space for $(L_1(M),L_\infty(M))$, and let $E^*(M)$ denote the dual space of $E(M)$ as a normed space. Then trace duality yields a canonical isometric operator $E^{\times}(M)\to E^*(M)$, but in general it may not be surjective. By \cite{DoddsSukochevIntegration}[Theorem 5.2.9, Proposition 5.3.2], it is surjective if and only if the norm of $E(M)$ is \textit{order\-/continuous}, i.e. if for every decreasing net $(x_\alpha)_\alpha$ of $E(M)_+$ such that $\inf_\alpha x_\alpha=0$ then $\inf_\alpha\|x_\alpha\|_{E(M)}=0$. Thus, if $E(M)$ is an exact interpolation space for $(L_1(M),L_\infty(M))$ with order\-/continuous norm, then $E(M)$ is a predual of $E^{\times}(M)$ so that $E^{\times}(M)$ inherits a canonical w*-topology.

The following result directly follows.

\begin{theo}
Let $E(M)$ be an exact interpolation space for $(L_1(M),L_\infty(M))$ with order-continuous norm. Then the subspace $L_1(M)\cap L_\infty(M)$ is norm-dense in $E(M)$ and \text{\upshape w*}-dense in $E^{\times}(M)$.
\end{theo}

\begin{coro}
Let $E(M)$ be an exact interpolation space for $(L_1(M),L_\infty(M))$ with order-continuous norm. Let $(p_\alpha)_\alpha$ be an increasing net of finite projections of $M$ such that $\sup_\alpha p_\alpha=1$. 
\begin{enumerate}
    \item If $x\in E(M)$, then the net $(p_\alpha xp_\alpha)_\alpha$ converges to $x$ in $E(M)$ for the norm topology. 
    \item If $y\in E^{\times}(M)$, then the net $(p_\alpha yp_\alpha)_\alpha$ converges to $y$ in $E^{\times}(M)$ for the \text{\upshape w*}\-/topology.
\end{enumerate}
\end{coro}
\begin{proof}
Let $x\in E(M)$ and $\epsilon>0$. As $\cup_{\alpha}p_\alpha Mp_\alpha$ is w*-dense in $M$, we deduce that $\cup_\alpha p_\alpha(L_1(M)\cap L_\infty(M))p_\alpha$ is $\sigma(E(M),E^{\times}(M))$\-/weakly dense in $E(M)$, and thus norm-dense in $E(M)$ because $E(M)$ has order-continuous norm. Hence, there is $\beta$ and $y\in L_1(M)\cap L_\infty(M)$ such that $\|x-p_\beta yp_\beta\|_{E(M)}<\epsilon$. Then, for all $\alpha\pg\beta$, we have
\begin{align*}
\|x-p_\alpha xp_\alpha\|_{E(M)}&=\|x-p_\beta yp_\beta+p_\beta yp_\beta-p_\alpha xp_\alpha\|_{E(M)}\\
&\pp\|x-p_\beta yp_\beta\|_{E(M)}+\|p_\alpha(p_\beta yp_\beta-x)p_\alpha\|_{E(M)}\\
&\pp2\|x-p_\beta yp_\beta\|_{E(M)}<2\epsilon
\end{align*}
which shows that $(p_\alpha xp_\alpha)_\alpha$ converges in norm to $x$. Now, if $y\in E^{\times}(M)$ then for every $x\in E(M)$ we get
\[\tau(xp_\alpha yp_\alpha)=\tau(p_\alpha xp_\alpha y)\underset{\alpha}{\to}\tau(xy)\]
as desired.
\end{proof}

\subsection{\texorpdfstring{$L_p$}{}-spaces and \texorpdfstring{$L_{p,\infty}$}{}-spaces}

Let $M$ be a semifinite von Neumann algebra, with trace denoted $\tau$. For $1<p<\infty$, and $x\in L_0(M)$, we set
\[\|x\|_{L_p(M)}:=\int_{0}^{\infty}\lambda_x(s)ps^{p-1}ds=\int_{0}^{\infty}\mu_x(s)^pds.\]
Then, for $1<p<\infty$, the \textit{$L_p$-space} 
\[L_p(M):=\big\{x\in L_0(M)\ :\ \|x\|_{L_p(M)}<\infty\big\}\]
is a subspace of $L_0(M)$ and becomes a complete normed space under $\|\cdot\|_{L_p(M)}$. 

\begin{theo}[\cite{DoddsSukochevIntegration}(Corollary 6.2.1)]
Let $1<p<\infty$. Then $L_p(M)$ is an exact interpolation space for $(L_1(M),L_\infty(M))$ with order\-/continuous norm, and $L_p(M)^{\times}=L_q(M)$ where $1<q<\infty$ is such that $1=\frac{1}{p}+\frac{1}{q}$.
\end{theo}

The following two results are classical.

\begin{theo}[Riesz-Thorin]
Let $1<p<\infty$. Then 
\[L_p(M)=[L_1(M),L_\infty(M)]_{1-1/p}\]
with equal norms.
\end{theo}

\begin{theo}
Let $1<p<\infty$. Then 
\[L_p(M)=(L_1(M),L_\infty(M))_{1-1/p,p,K}\]
with equivalent norms, with constants depending only on $p$.
\end{theo}


For $1<p<\infty$, and $x\in L_0(M)$, we set
\[\|x\|_{L_{p,\infty}(M)}:=\sup_{t>0}t^{-(1-1/p)}\int_{0}^{t}\mu_x(s)ds\]
Then, for $1<p<\infty$, the \textit{$L_{p,\infty}$-space} 
\[L_p(M):=\big\{x\in L_0(M)\ :\ \|x\|_{L_p(M)}<\infty\big\}\]
is a subspace of $L_0(M)$ and becomes a complete normed space under $\|\cdot\|_{L_{p,\infty}(M)}$. This is not the usual definition given in the literature, but it is equivalent because of the estimates
\[\sup_{t>0} t^{1/p}\mu_x(t)\pp\|x\|_{L_{p,\infty}(M)}\pp q\sup_{t>0}t^{1/p}\mu_x(t),\ \ \ \text{for}\ x\in L_0(M)\]
where $1<q<\infty$ is such that $p^{-1}+q^{-1}=1$. Finally, for $1\pp p<\infty$ we set
\[L_{p,1}(M):=L_{q,\infty}(M)^{\times}\]
where $1<q<\infty$ is such that $p^{-1}+q^{-1}=1$.

\begin{theo}[\cite{DoddsSukochevIntegration}(Theorem 6.3.8, Theorem 6.4.3)]
Let $1<p\pp\infty$, and let $1\pp q<\infty$ is such that $p^{-1}+q^{-1}=1$. Then $L_{q,1}(M)$ has an order\-/continuous norm and
\[L_{q,1}^{\times}(M)=L_{p,\infty}(M)\]
with equal norms.
\end{theo}

As a consequence, if $1<p\pp\infty$ then $L_{p,\infty}(M)$ is canonically equipped with a w*-topology as the dual of $L_{q,1}(M)$, where $1\pp q<\infty$ is such that $p^{-1}+q^{-1}=1$.

\begin{prop}\label{integration_amplification}
Let $N$ be a semifinite von Neumann algebra and $1<p\pp\infty$. If $x\in L_{p}(M)$ and $f\in L_{p,\infty}(N)$, then $x\otimes f\in L_{p,\infty}(M\bar{\otimes}N)$ with 
\[\|x\otimes f\|_{L_{p,\infty}(M\bar{\otimes}N)}\pp q\|x\|_{L_p(M)}\|f\|_{L_{p,\infty}(N)}.\]
where $1\pp q<\infty$ is such that $p^{-1}+q^{-1}=1$. Moreover, if $x\in L_p(M)$ then the bounded operator
\[\fonctbis{L_{p}(N)}{L_{p,\infty}(M\bar{\otimes}N)}{f}{x\otimes f}\]
is \text{\upshape w*}-continuous.
\end{prop}
\begin{proof}
The cas $p=\infty$ is classical, so we assume $p<\infty$. If $x\in L_p(M)$ and $f\in L_{p,\infty}(N)$, we have
\begin{align*}
\|x\otimes f\|_{L_{p,\infty}(M)}^p&\pp q\sup_{t>0}t^{1/p}\mu_{x\otimes f}(t)=q\sup_{t>0}t^p\lambda_{x\otimes f}(t)\\
&=q\sup_{t>0}t^p\int_{0}^{\infty}\lambda_{y}(t/\mu_x(s))ds\\
&=q\sup_{t>0}\int_{0}^{\infty}(t/\mu_x(s))^p\lambda_f(t/\mu_x(s))\mu_x(s)^pds\\
&\pp q\int_{0}^{\infty}\sup_{t>0}t^p\lambda_f(t)\mu_x(s)^pds\pp q\|x\|_{L_p(M)}^p\|f\|_{L_{p,\infty}(N)}^p.
\end{align*}
Now, fix $x\in L_p(M)$, and consider the operator $T_x:L_q(M)\odot L_{q,1}(N)\to L_{q,1}(N)$ such that \[T_x(y\otimes g)=\tau(xy)g,\ \ \ \text{for}\ y\in L_p(M),\ g\in L_{q,1}(N).\]
Let $\sigma$ denote the trace of $N$. If $w\in L_q(M)\odot L_{q,1}(N)$, then a direct computation shows that, for $f\in L_{p,\infty}(N)$ we have
\[\sigma(fT_x(w))=(\tau\otimes\sigma)((x\otimes f)w).\]
Thus, for $w\in (L_1\cap L_\infty)(M)\odot(L_1\cap L_\infty)(N)$ we get
\begin{align*}
\|T_x(w)\|_{L_{q,1}(N)}&=\sup_{\|f\|_{L_{p,\infty}(N)}\pp1}|\sigma(yT_x(w))|\\
&=\sup_{\|f\|_{L_{p,\infty}(N)}\pp1}|(\tau\otimes\sigma)((x\otimes f)w)|\\
&\pp\sup_{\|f\|_{L_{p,\infty}(N)}\pp1}\|x\otimes f\|_{L_{p,\infty}(M\bar{\otimes}N)}\|w\|_{L_{q,1}(M\bar{\otimes}N)}\\
&\pp p\|x\|_{L_p(M)}\|w\|_{L_{q,1}(M\bar{\otimes}N)}
\end{align*}
As the algebraic tensor product $(L_1\cap L_\infty)(M)\odot(L_1\cap L_\infty)(N)$ separates the points of $(L_1+L_\infty)(M\bar{\otimes}N)$, the above computation shows that $T_x$ extends to a bounded operator $L_{q,1}(M\bar{\otimes}N)\to L_{q,1}(N)$ whose dual coincides with the operator
\[\fonctbis{L_{p,\infty}(N)}{L_{p,\infty}(M\bar{\otimes}N)}{f}{x\otimes f}.\]
The proof is complete.
\end{proof}

\subsection{Conditional expectations}

Let $M$ be a semifinite von Neumann algebra and let $N$ be a \textit{semifinite von Neumann subalgebra} of $M$, i.e. a von Neumman subalgebra of $N$ such that the restriction of the trace to $N$ is again semifinite, and thus a n.s.f. trace on $N$. In that case, there is a (trace\-/preserving normal faithful) conditional expectation $E$ of $M$ onto $N$. Moreover, the conditional expectation $E$ extends to a contractive and idempotent compatible operator 
\[E:(L_1(M),L_\infty(M))\to(L_1(N),L_\infty(N))\]
that restricts to the identity on $L_1(N)+L_\infty(N)$. As a consequence, if $\mathcal{F}$ is an exact interpolation functor then $\mathcal{F}(L_1(N),L_\infty(N))$ is a subspace of $\mathcal{F}(L_1(M),L_\infty(M))$ and the inclusion operator $\mathcal{F}(L_1(N),L_\infty(N))\to\mathcal{F}(L_1(M),L_\infty(M))$ is isometric. As a consequence, if $\mathcal{F}$ is an exact interpolation functor, then the conditional expectation $E$ induces a canonical contractive idempotent operator \[E:\mathcal{F}(L_1(M),L_\infty(N))\to\mathcal{F}(L_1(M),L_\infty(M))\]
with range $\mathcal{F}(L_1(N),L_\infty(N))$.

\clearpage

\part{Interpolation of abstract Hardy spaces}

Let $M$ be a semifinite von Neumann algebra. 

\vspace{5pt}

\textbf{Notations.} If $P$ is any (self-adjoint) projection on $L_2(M)$ and if $E(M)$ is any exact interpolation space for the compatible couple $(L_1(M),L_\infty(M))$, we define $H_E(P)$ to be the $\sigma(E(M),E^{\times}(M))$\-/weak closure of the subspace
\[E(M)\cap P(L_2(M))\]
in $E(M)$. Let us denote $H_p(P):=H_{L_p}(P)$ for $1\pp p\pp\infty$. By definition, $H_2(P)=P(L_2(M))$. Let us also denote $H_{p,1}(P):=H_{L_{p,1}}(P)$ and $H_{p,\infty}(P):=H_{L_{p,\infty}}(P)$ for $1<p<\infty$. They are the \textit{abstract Hardy spaces} associated with $P$.

\begin{rem}
If $P$ is a projection on $L_2(M)$ with complement projection $P^{\bot}$, and if $E(M)$ is an exact interpolation space for $(L_1(M),L_\infty(M))$, then from the definitions it is clear that $H_E(P)$ and $H_{E^{\times}}(P)$ are respectively included in the orthogonal of $H_{E^{\times}}(P^\bot)$ and $H_E(P^{\bot})$ w.r.t. trace duality between $E(M)$ and $E^{\times}(M)$.
\end{rem}

\section{The technical assumptions}

Let $M$ be a semifinite von Neumann algebra.

\begin{mdframed}[backgroundcolor=black!10,rightline=false,leftline=false,topline=false,bottomline=false,skipabove=10pt]
A projection $P$ on $L_2(M)$ is said to satisfy the \textit{technical assumptions} if
\vspace{3pt}
\begin{enumerate}[nosep]
    \item[(1)] If $E(M)$ is an exact interpolation space for $(L_1(M),L_\infty(M))$ with order\-/continuous norm, then the orthogonals of $H_E(P)$, $H_{E^{\times}}(P)$ w.r.t. trace duality coincide respectively with $H_E(P^{\bot})$, $H_{E^{\times}}(P^{\bot})$.
    \item[(2)] There is a subspace $H(P)$ of $(L_1+L_\infty)(M)$, which is weakly-closed in $(L_1+L_\infty)(M)$ (w.r.t. trace duality between $(L_1+L_\infty)(M)$ and $(L_1\cap L_\infty)(M)$), and such that, if $E(M)$ is an exact interpolation space for $(L_1(M),L_\infty(M))$ with order\-/continuous norm, then we have 
\[H_E(P)=H(P)\cap E(M),\ \ \ \ H_{E^{\times}}(P)=H(P)\cap E^{\times}(M).\]
    \item[(3)] If $E(M)$ is an exact interpolation space for $(L_1(M),L_\infty(M))$ with order\-/continuous norm, then the subspace
\[(L_1\cap L_\infty)(M)\cap P(L_2(M))\]
is norm-dense in $H_E(P)$ and w*-dense in $H_{E^{\times}}(P)$.
\end{enumerate}
\end{mdframed}
\vspace{5pt}

\begin{prop}
Let $P$ be a projection on $L_2(M)$ which satisfies the technical assumptions. Then the complement projection $P^{\bot}$ on $L_2(M)$ also satisfies the technical assumptions.
\end{prop}
\begin{proof}
As $P$ satisfies (1), by the bipolar theorem it is clear that $P^{\bot}$ satisfies (1). 

For (2), let $H(P^{\bot})$ denote the orthogonal of $(L_1\cap L_\infty)(M)\cap P(L_2(M))$ w.r.t. trace duality between $(L_1+L_\infty)(M)$ and $(L_1\cap L_\infty)(M)$. Then $H(P^{\bot})$ is indeed weakly-closed in $(L_1+L_\infty)(M)$. Let $E(M)$ be an exact interpolation space for $(L_1(M),L_\infty(M))$ with order\-/continuous norm. As $H_E(P^{\bot})$ is orthogonal to $H_{E^{\times}}(P)$, and because $H_{E^{\times}}(P)$ contains $(L_1\cap L_\infty)(M)\cap P(L_2(M))$, we deduce that $H_E(P^{\bot})$ is included in $H(P^{\bot})$, and thus in $E(M)\cap H(P^{\bot})$. For the converse inclusion, let $x\in E(M)\cap H(P^{\bot})$. As $P$ satisfies (3), we deduce that $x$ is orthogonal to $H_{E^{\times}}(P)$. As $P$ satisfies (1), we deduce that $x\in H_E(P^{\bot})$. Thus $H_E(P^{\bot})=E(M)\cap H(P^{\bot})$ as desired. In a similar way, we prove that $H_{E^{\times}}(P^{\bot})=E^{\times}(M)\cap H(P^{\bot})$. 

It remains to justify (3). Note that $(L_1\cap L_\infty)(M)\cap P^{\bot}(L_2(M))$ is weakly closed in $(L_1\cap L_\infty)(M)$, as a consequence of the fact that $P$ is weakly-continuous for the weak topology of $L_2(M)$. Thus, as $P$ satisfies (1), by the bipolar theorem it suffices to show that, if $x\in E^{\times}(M)$ resp. $x\in E(M)$ is orthogonal to $(L_1\cap L_\infty)(M)\cap P^{\bot}(L_2(M))$, then $x\in H_{E^\times}(P)$ resp. $x\in H_E(P)$. Let $x\in E^{\times}(M)$ resp. $x\in E(M)$ orthogonal to $(L_1\cap L_\infty)(M)\cap P^{\bot}(L_2(M))$. Then $x$ is in the biorthogonal of $H(P)$ w.r.t. trace duality between $(L_1+L_\infty)(M)$ and $(L_1\cap L_\infty)(M)$. By the bipolar theorem, and because $H(P)$ is, by assumption, weakly closed in $(L_1+L_\infty)(M)$, we deduce that $x\in H(P)$. Thus $x\in E^{\times}(M)\cap H(P)$ resp. $x\in E^{\times}(M)\cap H(P)$. As $P$ satisfies (2), we finally deduce that $x\in H_{E^{\times}}(P)$ resp. $x\in H_E(P)$, as desired.
\end{proof}

The following result is a direct consequence of Pisier's theorem (Theorem \ref{preliminaries_Pisier}).

\begin{prop}\label{technical_assumptions_Pisier}
Let $P$ be a projection on $L_2(M)$ which satisfies the technical assumptions, and let $E_0(M)$ $E_1(M)$ be two exact interpolation spaces for $(L_1(M),L_\infty(M))$ with order\-/continuous norm. If the subcouple $(H_{E_0}(P),H_{E_1}(P))$ is quasi-complemented in $(E_0(M),E_1(M))$ with constant $C$, then the subcouple $(H_{E_0^{\times}}(P^{\bot}),H_{E_1^{\times}}(P^{\bot}))$ is quasi-complemented in $(E_0^{\times}(M),E_1^{\times}(M))$ with a constant depending only on $C$.
\end{prop}

\begin{prop}\label{technical_assumptions_Kclosedness}
Let $P$ be a projection on $L_2(M)$ which satisfies the technical assumptions, let $E_0(M)$, $E_1(M)$ be two exact interpolation spaces for $(L_1(M),L_\infty(M))$ such that the subcouple $(H_{E_0}(P),H_{E_1}(P))$ is quasi-complemented in $(E_0(M),E_1(M))$ with constant $C$, and let $\Phi$ be a $K$-parameter space such that the exact interpolation space
\[E_\theta(M):=K_\Phi(E_0(M),E_1(M))\]
has order-continuous norm. Then
\[H_{E_\theta}(P)=K_\Phi(E_0(M),E_1(M))\]
with equivalent norms, with constants depending only on $C$.
\end{prop}
\begin{proof}
By Proposition \ref{preliminaries_Kclosedness}, it suffices to check that
\[(H_{E_0}(P)+H_{E_1}(P))\cap E_\theta(M)\]
is a dense subspace of $H_{E_\theta}(P)$. It is a subspace because of point (2) of the technical assumptions, and it is dense because of point (3) of the technical assumptions.
\end{proof}

\begin{prop}\label{technical_assumptions_Jclosedness}
Let $P$ be a projection on $L_2(M)$ which satisfies the technical assumptions, and let $E_0(M)$, $E_1(M)$ be two exact interpolation spaces for $(L_1(M),L_\infty(M))$ such that the subcouple $(H_{E_0}(P),H_{E_1}(P))$ is quasi-complemented in $(E_0(M),E_1(M))$ with constant $C$. Let $\Phi$ be a $J$-parameter space such that the exact interpolation space
\[E_\theta^{\times}(M):=J_\Phi(E_0^{\times}(M),E_1^{\times}(M))\]
is the K\"othe dual of an exact interpolation space $E_\theta(M)$ with order\-/continuous norm. In addition, assume that $E_0(M)$, $E_1(M)$ has both order-continuous norm 
Then, the subcouple $(H_{E_0^\times}(P^{\bot}),H_{E_1^{\times}}(P^{\bot}))$ is normal in $(E_0^{\times}(M),E_1^{\times}(M))$, we have 
\[E_\theta^{\times}(M)\cap(H_{E_0^{\times}}(P^{\bot})+H_{E_1^{\times}}(P^{\bot}))=H_{E_\theta^{\times}}(P^{\bot})\]
and
\[E_\theta^{\times}(M)/H_{E_\theta^{\times}}(P^{\bot})=J_\Phi(E_0^{\times}(M)/H_{E_0^{\times}}(P^{\bot}),E_1^{\times}(M)/H_{E_1^{\times}}(P^{\bot}))\]
with equivalent norms, with constants depending only on $C$.
\end{prop}
\begin{proof}
As $P^{\bot}$ satisfies the technical assumptions, we see that $E_\theta^{\times}(M)\cap(H_{E_0^{\times}}(P^{\bot})+H_{E_1^{\times}}(P^{\bot}))$ is a w*-dense subspace of $H_{E_\theta^{\times}}(P^{\bot})$. Thus, by the Krein-Smulian theorem, it suffices to check that the unit ball of $E_\theta^{\times}(M)\cap (H_{E_0^{\times}}(P^{\bot})+H_{E_1^{\times}}(P^{\bot}))$ is w*-closed in $H_{E_\theta^{\times}}(P^{\bot})$. Let $(x_\alpha)_\alpha$ be a net in the unit ball of $E_\theta^{\times}(M)\cap (H_{E_0^{\times}}(P^{\bot})+H_{E_1^{\times}}(P^{\bot}))$ and let $x\in H_{E_\theta^{\times}}(P^{\bot})$ such that $(x_\alpha)_\alpha$ converges to $x$ for the w*-topology of $E_\theta^{\times}(M)$. By Proposition \ref{technical_assumptions_Pisier}, we know that $(H_{E_0^{\times}}(P^{\bot}),H_{E_1^{\times}}(P^{\bot}))$ is $K$-closed in $(E_0^{\times}(M),E_1^{\times}(M))$ with constant $C$. As $x_\alpha\in H_{E_0^{\times}}(P^{\bot})+H_{E_1^{\times}}(P^{\bot})$, we deduce that we can write
\[x_\alpha=y_\alpha+z_\alpha\]
where $y_\alpha\in H_{E_0^{\times}}(P^{\bot})$, $z_\alpha\in H_{E_1^{\times}}(P^{\bot})$, with
\[\|y_\alpha\|_{E_0^{\times}}+\|z_\alpha\|_{E_1^{\times}}\pp 2C\|x_\alpha\|_{E_0^{\times}+E_1^{\times}}.\]
By definition of $E_\theta^{\times}(M)$, we have $\|x_\alpha\|_{E_0^{\times}(M)+E_1^{\times}(M)}\pp\|x_\alpha\|_{E_\theta^{\times}(M)}\pp1$. Thus, the nets $(y_\alpha)_\alpha$ and $(z_\alpha)_\alpha$ are bounded in respectively $E_0^{\times}(M)$ and $E_1^{\times}(M)$. As $E_0(M)$, $E_1(M)$ have an order\-/continuous norm, by the Banach-Alaoglu theorem, we deduce that the nets $(y_\alpha)_\alpha$ $(z_\alpha)_\alpha$ respectively admit a cluster point $y\in H_{E_0^{\times}}(P^{\bot})$ and $z\in H_{E_1}^{\times}(P^{\bot})$.




As $x=y+z$, we obtain $x\in H_{E_0^{\times}}(P^{\bot})+H_{E_1^{\times}}(P^{\bot})$, as desired. Finally, by Proposition \ref{preliminaries_Jclosedness}, we deduce that we have
\[E_\theta^{\times}(M)/H_{E_\theta}(P^{\bot})=J_\Phi(E_0^{\times}(M)/H_{E_0}(P^{\bot}),E_1^{\times}(M)/H_{E_1}(P^{\bot}))\]
with equivalent norms, with constants depending only on $C$.
\end{proof}

The following result is a direct consequence of Kislyakov/Xu's theorem (Theorem \ref{preliminaries_KislyakovXu}) together with Theorem \ref{technical_assumptions_Kclosedness}.

\begin{prop}\label{technical_assumptions_KislyakovXu}
Let $P$ be a projection on $L_2(M)$ which satisfies the technical assumptions, and let $E_0(M)$, $E_{\theta_0}(M)$, $E_{\theta_1}(M)$, $E_1(M)$ be four exact interpolation spaces for $(L_1(M),L_\infty(M))$, such that $E_{\theta_0}(M)$, $E_{\theta_1}(M)$ have both order\-/continuous norm, and such that 
\[E_{\theta_0}(M)=(E_0(M),E_{\theta_1}(M))_{\eta_0,p_0,K},\ \ \ E_{\theta_1}(M)=(E_{\theta_0}(M),E_1(M))_{\eta_1,p_1,K}\]
for given $0<\eta_0,\eta_1<1$ and $1\pp p_0,p_1\pp\infty$. In addition, we make the following two assumptions.
\begin{enumerate}[nosep]
    \item the subcouple $(H_{E_0}(P),H_{E_{\theta_1}}(P))$ is quasi-complemented in $(E_0(M),E_{\theta_1}(M))$ with constant $C_0$.
    \item the subcouple $(H_{E_{\theta_1}}(P),H_{E_1}(P))$ is quasi-complemented in $(E_{\theta_1}(M),E_1(M))$ with constant $C_1$.
\end{enumerate}
Then the subcouple $(H_{E_0}(P),H_{E_1}(P))$ is quasi-complemented in $(E_0(M),E_1(M))$ with a constant depending only on $C,\eta_0,\eta_1$.
\end{prop}

\section{Jones-type projections}

Let $M$ be a semifinite von Neumann algebra.

\begin{mdframed}[backgroundcolor=black!10,rightline=false,leftline=false,topline=false,bottomline=false,skipabove=10pt]
A projection $P$ on $L_2(M)$ is said to be of \textit{Jones-type} if
\vspace{3pt}
\begin{enumerate}[nosep]
    \item[(1)] $P$ satisfies the technical assumptions.
    \item[(2)] If $1<p<\infty$, then $P$ is $L_p(M)$-bounded with a constant depending only on $p$.
    \item[(3)] The subcouples $(H_1(P),H_2(P))$ and $(H_1(P^{\bot}),H_2(P^{\bot}))$ are both $K$-closed in the compatible couple $(L_1(M),L_2(M))$ with a universal constant.
\end{enumerate}
\end{mdframed}
\vspace{5pt}

\begin{rem}
It is clear that if $P$ is a projection $L_2(M)$ which is of Jones-type, then the complement projection $P^{\bot}$ is also of Jones-type.
\end{rem}

\begin{theo}
Let $P$ be a Jones-type projection on $L_2(M)$. Then the subcouple $(H_1(P),H_\infty(P))$ is quasi\-/complemented in the compatible couple $(L_1(M),L_\infty(M))$ with a universal constant.
\end{theo}
\begin{proof}
By property (2), the subcouple $(H_{3/2}(P),H_4(P))$ is complemented and thus quasi\-/complemented in $(L_{3/2}(M),L_4(M))$ with a universal constant, so that we can apply Proposition \ref{technical_assumptions_KislyakovXu} with the subquadruple $(H_1(P),H_{3/2}(P),H_2(P),H_4(P))$ of $(L_1(M),L_{3/2}(M),L_2(M),L_4(M))$ to deduce that the subcouple $(H_1(P),H_4(P))$ is quasi\-/complemented in $(L_1(M),L_4(M))$ with a universal constant. As $(H_1(P^{\bot}),H_2(P^{\bot}))$ is quasi-complemented in $(L_1(M),L_2(M))$ with a universal constant, by Proposition \ref{technical_assumptions_Pisier} together with property (1) and (3) the subcouple $(H_2(P),H_\infty(P))$ is quasi\-/complemented in $(L_2(M),L_\infty(M))$ with a universal constant, so that we can again apply Proposition \ref{technical_assumptions_KislyakovXu} with the subquadruple $(H_1(P),H_2(P),H_4(P),H_\infty(P))$ of $(L_1(M),L_2(M),L_4(M),L_\infty(M))$ to deduce that the subcouple $(H_1(P),H_\infty(P))$ is quasi\-/complemented in $(L_1(M),L_\infty(M))$ with a universal constant. 
\end{proof}

By Proposition \ref{technical_assumptions_Kclosedness}, we deduce the following result.

\begin{theo}
Let $P$ be a Jones-type projection on $L_2(M)$, and let $\Phi$ be a $K$-parameter space such that the exact interpolation space
\[E(M):=K_\Phi(L_1(M),L_\infty(M))\]
has order-continuous norm. Then
\[H_E(P)=K_\Phi(H_1(P),H_{\infty}(P))\]
with equivalent norms, with universal constants.
\end{theo}

\begin{coro}
Let $P$ be a Jones-type projection on $L_2(M)$, and let $0<\theta<1$. Then
\[H_{p_\theta}(P)=(H_1(P),H_{\infty}(P))_{\theta,p_\theta,K}\]
with equivalent norms, with constants depending only on $\theta$, where $1/p_\theta=1-\theta$.
\end{coro}

\section{Pisier's method}

\subsection{Stein's interpolation theorem}

For future purposes, we require a version of Stein's interpolation theorem for analytic families of operators that holds in the setting of dual spaces equipped with their w*\-/topology. Since we were unable to find such a statement in the literature, we provide a proof in a slightly weaker form, which suffices for the purposes of this paper. 

\vspace{10pt}

Let $E^*$ be a complete normed space equipped with a predual $E$, so that $E^*$ is equipped with a w*\-/topology. A function $f:\S\to E^*$ is said to be w*\textit{-holomorphic} if the function $\S\to\C$, $z\mapsto\bra u,f(z)\ket$ is holomorphic, for every $u\in E$. 

Let $(F_0^*,F_1^*)$ be a compatible couple equipped with a predual, i.e. a regular compatible couple $(F_0,F_1)$ together with an isometric isomorphism of compatible couples between $(F_0^*,F_1^*)$ and the dual of $(F_0,F_1)$ as defined in the preliminary section. Let $\mathcal{F}_*(F_0^*,F_1^*)$ denote the space of functions $f:\overline{\S}\to F_0^*+F_1^*$ w*\-/continuous on $\overline{\S}$ and w*\-/holomorphic on $\S$ (when $F_0^*+F_1^*$ is equipped with the w*\-/topology coming from the pairing with $F_0\cap F_1$), such that $f(j+it)\in F_j^*$ for $t\in\R$, $j\in\{0,1\}$, and such that the functions $\R\to F_j^*$, $t\mapsto f(j+it)$ are w*\-/continuous and bounded, for $j\in\{0,1\}$. It is clear that $\mathcal{F}(F_0^*,F_1^*)$ is a subspace of $\mathcal{F}_*(F_0^*,F_1^*)$. For $f\in\mathcal{F}_*(F_0^*,F_1^*)$, we set
\[\|f\|_{\mathcal{F}(F_0^*,F_1^*)}:=\max_{j\in\{0,1\}}\sup_{t\in\R}\|f(j+it)\|_{F_j^*}.\]
The proof of the following proposition is straightforward.

\begin{prop}\label{steintheorem_prop1}
Let $(F_0^*,F_1^*)$ and $(G_0^*,G_1^*)$ be two compatible couples equipped with preduals. Let $T:(G_0^*,G_1^*)\to(F_0^*,F_1^*)$ be a compatible bounded operator such that the operators $T:G_0^*\to F_0^*$ and $T:G_1^*\to F_1^*$ are \text{\upshape w*}\-/continuous. If $f\in\mathcal{F}_*(G_0^*,G_1^*)$, then the function 
\[\fonctbis{\B}{F_0^*+F_1^*}{z}{T(f(z))}\]
belongs to $\mathcal{F}_*(F_0^*,F_1^*)$.
\end{prop}

\begin{mdframed}[skipabove=10pt]
\begin{theo}[w*\-/Stein's interpolation theorem]\label{steintheorem_theo1}
Let $(E_0,E_1)$ be a compatible couple and let $(F_0^*,F_1^*)$ be a compatible couple equipped with a predual. Let $(T_z)_{z\in\overline{\S}}$ be a family of operators $E_0\cap E_1\to F_0^*+F_1^*$ satisfying the following two conditions. 
\begin{itemize}
    \item For every $u\in E_0\cap E_1$, the function 
\[\fonctbis{\S}{F_0^*+F_1^*}{z}{T_z(u)}\]
belongs to $\mathcal{F}_*(F_0^*,F_1^*)$. 
    \item For every $u\in E_0\cap E_1$, $t\in\R$ and $j\in\{0,1\}$, we have
    \[\|T_{j+it}(u)\|_{F_j^*}\pp\|u\|_{E_j}.\]
\end{itemize}
Then, for $u\in E_0\cap E_1$ and $0<\theta<1$ we have $T_{\theta}(u)\in(F_0^*,F_1^*)_{\theta,\infty,K}$ with \[\|T_{\theta}(u)\|_{(F_0^*,F_1^*)_{\theta,\infty,K}}\pp\|u\|_{[E_0,E_1]_\theta}.\]
\end{theo}
\end{mdframed}
\vspace{5pt}


For the proof, we need a couple of preliminary results. If $(E_0,E_1)$ is a compatible couple, let $\mathcal{F}_\tau(E_0,E_1)$ denote the set of functions $f:\overline{\S}\to E_0+E_1$ which can be written as a finite sum $f=\sum_ng_n\otimes u_n$ where $u_n\in E_0\cap E_1$, and where $g_n:\overline{\S}\to\C$ is a bounded continuous function, holomorphic on $\S$, such that the continuous function $\R\to\C$, $t\mapsto g(j+it)$ vanishes at infinity for $j\in\{0,1\}$. It is clear that $\mathcal{F}_\tau(E_0,E_1)$ is a subspace of $\mathcal{F}(E_0,E_1)$.

\begin{lemm}[\cite{StafneySteinTheorem}[Lemma 2.5]\label{steintheorem_lemm1}
Let $(E_0,E_1)$ be a compatible couple, $0<\theta<1$ and $u\in E_0\cap E_1$ such that $\|u\|_{[E_0,E_1]_\theta}<1$. Then there is $f\in\mathcal{F}_\tau(E_0,E_1)$ such that $u=f(\theta)$ and 
$\|f\|_{\mathcal{F}(E_0,E_1)}<1$.
\end{lemm}

\begin{lemm}\label{steintheorem_lemm2}
Let $(F_0^*,F_1^*)$ be a compatible couple equipped with a predual $(F_0,F_1)$, $f\in\mathcal{F}_*(F_0^*,F_1^*)$ and $0<\theta<1$. Then $f(\theta)\in(F_0^*,F_1^*)_{\theta,\infty}$ with
\[\|f(\theta)\|_{(F_0^*,F_1^*)_{\theta,\infty,K}}\pp\|f\|_{\mathcal{F}(F_0^*,F_1^*)}.\]
\end{lemm}
\begin{proof}
Fix $t>0$. If $u\in F_0\cap F_1$, then by Hadamard's three-lines theorem applied to the function $\overline{\S}\to\C$, $z\mapsto\bra u,f(z)\ket$, we have
\[|\bra u,f(\theta)\ket|\pp\sup_{s\in\R}|\bra u,f(is)\ket|^{1-\theta}\sup_{s\in\R}|\bra u,f(1+is)\ket|^{\theta}.\]
Now, if we denote $C:=\|f\|_{\mathcal{F}(F_0^*,F_1^*)}$, for $s\in\R$ we have
\[|\bra u,f(is)\ket|\pp\|u\|_{F_0}\|f(j+is)\|_{F_0^*}\pp C\|u\|_{F_0},\]
\[|\bra u,f(1+is)\ket|\pp\|u\|_{F_1}\|f(j+is)\|_{F_1^*}\pp C\|u\|_{F_1}=Ct\|u\|_{t^{-1}F_1}.\]
Thus, by combining the previous estimates, we get
\[|\bra u,f(\theta)\ket|\pp C\|u\|_{F_0}^{1-\theta}t^{\theta}\|u\|_{t^{-1}F_1}^{\theta}\pp Ct^{\theta}\|u\|_{F_0\cap t^{-1}F_1}.\]
Finally, we obtain 
\begin{align*}
K_t(f(\theta),F_0^*,F_1^*)&=\|f(\theta)\|_{F_0^*+tF_1^*}\\
&=\sup_{\protect\substack{u\in F_0\cap F_1\\\|u\|_{F_0\cap t^{-1}F_1}\pp1}}|\bra u,f(\theta)\ket|\\
&\pp\sup_{\protect\substack{u\in F_0\cap F_1\\\|u\|_{F_0\cap tF_1}\pp1}}Ct^{\theta}\|u\|_{F_0\cap t^{-1}F_1}\\
&\pp Ct^{\theta}.
\end{align*}
This shows indeed that $f(\theta)\in(F_0^*,F_1^*)_{\theta,\infty,K}$ with $\|f(\theta)\|_{(F_0^*,F_1^*)_{\theta,\infty,K}}\pp C$.
\end{proof}

Now we are able to complete the proof of Theorem \ref{steintheorem_theo1}.

\begin{proof}[Proof of w*\-/Stein's interpolation theorem]
Fix $u\in E_0\cap E_1$ and $0<\theta<1$ with $\|u\|_{[E_0,E_1]_\theta}<1$. By Lemma \ref{steintheorem_lemm1}, we have $u=f(\theta)$ with $f=\sum_nf_n\otimes u_n\in\mathcal{F}_\tau(E_0,E_1)$ such that 
\[\|f\|_{\mathcal{F}(E_0,E_1)}<1.\]
Let $g:\overline{\S}\to F_0^*+F_1^*$ be the function such that $g(z)=T_z(f(z))=\sum_nf_n(z)T_z(u_n)$ for $z\in\overline{\S}$. By hypothesis, it is clear that $g$ belongs to $\mathcal{F}_*(F_0^*,F_1^*)$. Moreover, if $t\in\R$ then by hypothesis
\[\|g(j+it)\|_{F_j^*}=\|T_{j+it}(f(j+it))\|_{F^*_j}\pp\|f(j+it)\|_{E_j}\pp\|f\|_{\mathcal{F}(E_0,E_1)}\]
and thus
\[\|g\|_{\mathcal{F}(F_0^*,F_1^*)}\pp\|f\|_{\mathcal{F}(E_0,E_1)}<1.\]
As we have
\[g(\theta)=T_\theta(f(\theta))=T_\theta(u),\]
by Lemma \ref{steintheorem_lemm2} we deduce that $T_\theta(u)\in(F_0^*,F_1^*)_{\theta,\infty,K}$ with \[\|T_\theta(u)\|_{(F_0^*,F_1^*)_{\theta,\infty,K}}\pp\|g\|_{\mathcal{F}(F_0^*,F_1^*)}<1.\]
The proof is complete.
\end{proof}

\subsection{Haagerup's embedding}

 
Let $N$ denote the commutative von Neumann algebra $L_\infty(\R)$ equipped with the (n.s.f) trace $\sigma$ such that
\[\sigma(f)=\int_{-\infty}^{\infty}f(t)e^tdt,\ \ \ \text{for}\ f\in N_+.\]
In the sequel we fix an exponent $1<q<\infty$ and we consider the dual exponent $1<p<\infty$, i.e. such that 
\[\frac{1}{p}+\frac{1}{q}=1.\]
For $z\in\overline{\B}$, let $p_z\in[1,p]$ and $q_z\in[q,\infty]$ be the exponents such that
\[\frac{1}{p_z}=\frac{1-\re(z)}{p}+\re(z),\ \ \ \ \ \ \ \ \ \ \ \frac{1}{q_z}=\frac{1-\re(z)}{q}.\]  
Finally, for $z\in\overline{\B}$, we consider the continuous function $f_z\in C(\R)$ such that
\[f_z(s)=e^{-s(1-z)/q},\ \ \ \ \ \ \text{for}\ s\in\R.\]
A direct computation shows that
\[\lambda_{f_z}(s)=s^{-q_z},\ \ \ \mu_{f_z}(s)=s^{-1/q_z}\ \ \text{for}\ s>0.\]
Thus, we get
\begin{align*}
\|f_z\|_{L_{q_z,\infty}(N)}&=\sup_{t>0}t^{-1/p_z}K_t(f_z,L_1(M\bar{\otimes}N),L_\infty(M\bar{\otimes}N))\\
&=\sup_{t>0}t^{-1/p_z}\int_{0}^{t}\mu_{f_z}(s)ds\\
&=\sup_{t>0}t^{-1/p_z}\int_{0}^{t}s^{-1/q_z}ds\\
&=p_z.
\end{align*}
and in particular $f_z\in L_{q_z,\infty}(N)$. In the sequel, the compatible couple $(L_{q,\infty}(N),L_\infty(N))$ is equipped with its canonical predual coming from its canonical pairing with the compatible couple $(L_{p,1}(N),L_1(N))$ so that we can consider the space $\mathcal{F}_*(L_{q,\infty}(N),L_\infty(N))$ as defined in the previous paragraph.



\begin{lemm}\label{amplification_lemma1}
The function
\[\fonctbis{\overline{\S}}{L_{q,\infty}(N)+L_\infty(N)}{z}{f_z}\]
is well-defined and belongs to $\mathcal{F}_*(L_{q,\infty}(N),L_\infty(N))$.
\end{lemm}
\begin{proof}
If $z\in\overline{\S}$, then by the reiteration theorem for the real method, we have
\[L_{q_z,\infty}(N)=(L_{q,\infty}(N),L_\infty(N))_{\re(z),\infty}.\]
As $f_z\in L_{q_z,\infty}(N)$, we deduce that $f_z\in L_{q,\infty}(N)+L_\infty(N)$. Thus, the function $\overline{\S}\to L_{q,\infty}(N)+L_\infty(N)$, $z\mapsto f_z$ is indeed well-defined. Moreover, the functions $\R\to L_{q,\infty}(N)$, $t\mapsto f_{it}$ and $\R\to L_\infty(N)$, $t\mapsto f_{1+it}$ are clearly norm-bounded. Now, if $g\in L_1(N)$, then for every $t\in\R$ we have
\[\sigma(f_{1+it}g)=\int_{-\infty}^{\infty}e^{ist/q}g(s)d\sigma(s).\]
Hence the function $\R\to\C$, $t\mapsto\sigma(f_{1+it}g)$ is continuous. Now, if $g\in L_{p,1}(N)$, then for every $t\in\R$ we have
\[\sigma(f_{it}g)=\int_{-\infty}^{\infty}e^{-s(1-it)/q}g(s)d\sigma(s)=\int_{-\infty}^{\infty}e^{ist/q}f_0(s)g(s)d\sigma(s).\]
As $f_0\in L_{q,\infty}(N)=L_{p,1}(N)^{\times}$, we have $f_0g\in L_1(N)$ and thus the function $\R\to\C$, $t\mapsto\sigma(f_{it}g)$ is continuous. Hence, the functions $\R\to L_{q,\infty}(N)$, $t\mapsto f_{it}$ and $\R\to L_\infty(N)$, $t\mapsto f_{1+it}$ are both w*\-/continuous. In a similar way, it is easy to see that the function 
\[\fonctbis{\overline{\S}}{L_{q,\infty}(N)+L_\infty(N)}{z}{f_z}\]
is w*\-/continuous on $\overline{\S}$ and w*\-/holomorphic on $\S$.
\end{proof}

Now, let $M$ be a semifinite von Neumann algebra.

If $z\in\overline{\S}$, as $f_z\in L_{q_z,\infty}(N)$, by Proposition \ref{integration_amplification} we know that the operator
\[\fonct{F_z}{L_{q_z}(M)}{L_{q_z,\infty}(M\bar{\otimes}N)}{x}{x\otimes f_z}.\]
is well-defined. In the sequel, the compatible couple $(L_{q,\infty}(M\bar{\otimes}N),L_\infty(M\bar{\otimes}N))$ is equipped with its predual coming from its canonical pairing with the compatible couple $(L_{p,1}(M\bar{\otimes}N),L_1(M\bar{\otimes}N))$, so that we can consider the space $\mathcal{F}_*(L_{q,\infty}(M\bar{\otimes}N),L_\infty(M\bar{\otimes}N))$ as defined in the previous paragraph.

\begin{mdframed}[skipabove=10pt]
\begin{theo}\label{amplification_theo1}
The family of operators $(F_z)_{z\in\overline{\S}}$ satisfies the following conditions.
\begin{itemize} 
    \item for every $x\in(L_q\cap L_\infty)(M)$, the function
\[\fonctbis{\overline{\S}}{L_{q,\infty}(M\bar{\otimes}N)+L_\infty(M\bar{\otimes}N)}{z}{F_z(x)}\]
 belongs to $\mathcal{F}_*(L_{q,\infty}(M\bar{\otimes}N),L_\infty(M\bar{\otimes}N))$,
    \item for every $z\in\overline{\S}$ and $x\in L_{q_z}(M)$ we have
\[\|F_z(x)\|_{L_{q_z,\infty}(M\bar{\otimes}N)}=p_z\|x\|_{L_{q_z}(M)}.\]
\end{itemize}
\end{theo}
\end{mdframed}
\begin{proof}

$\bullet$ Let $x\in(L_q\cap L_\infty)(M)$. Because of Proposition \ref{integration_amplification}, we know that the operators
\[\fonctbis{L_{q,\infty}(M)}{L_{q,\infty}(M\bar{\otimes}N)}{f}{x\otimes f}\ \ \ \ \ \ \ \text{and}\ \ \ \ \fonctbis{L_\infty(N)}{L_\infty(M\bar{\otimes}N)}{f}{x\otimes f}\]
are both w*\-/continuous. By using Lemma \ref{amplification_lemma1} together with Proposition \ref{steintheorem_prop1} we deduce that the function
\[\fonctbis{\overline{\S}}{L_{q,\infty}(M\bar{\otimes}N)+L_\infty(M\bar{\otimes}N)}{z}{x\otimes f_z}\]
 belongs to $\mathcal{F}_*(L_{q,\infty}(M\bar{\otimes}N),L_\infty(M\bar{\otimes}N))$. 
 
$\bullet$ Let $z\in\overline{\B}$. We can assume $\re(z)<1$, otherwise the result is trivial. Then, if $x\in L_{q_z}(M)$ we have
\begin{align*}
\lambda_{x\otimes f_z}(s)&=\int_{0}^{\infty}\lambda_{f_z}(s/\mu_x(t))dt=\int_{0}^{\infty}(s/\mu_x(t))^{-q_z}dt\\
&=s^{-q_z}\int_{0}^{\infty}\mu_x(t)^{q_z}dt=s^{-q_z}\|x\|_{L_{q_z}(M)}^{q_z}.
\end{align*}
thus, we get
\[\mu_{x\otimes f_z}(s)=s^{-1/q_z}\|x\|_{L_{q_z}(M)}\]
and then
\begin{align*}
\|F_z(x)\|_{L_{q_z,\infty}(M\bar{\otimes}N)}&=\|x\otimes f_z\|_{L_{q_z,\infty}(M\bar{\otimes}N)}\\
&=\sup_{t>0}t^{-1/p_z}K_t(x\otimes f_z,L_1(M\bar{\otimes}N),L_\infty(M\bar{\otimes}N))\\
&=\sup_{t>0}t^{-1/p_z}\int_{0}^{t}\mu_{x\otimes f_z}(s)ds\\
&=\sup_{t>0}t^{-1/p_z}\|x\|_{L_{q_z}(M)}\int_{0}^{t}s^{-1/q_z}ds\\
&=\sup_{t>0}t^{-1/p_z}(1-1/q_z)^{-1}t^{1-1/q_z}\|x\|_{L_{q_z}(M)}\\
&=p_z\|x\|_{L_{q_z}(M)}.
\end{align*}
\end{proof}

\begin{rem}
Let $1<q\pp\infty$. The range of the operator 
\[\fonct{F_0}{L_q(M)}{L_{q,\infty}(M\bar{\otimes}N)}{x}{x\otimes f_0}\]
is the so-called \textit{Haagerup's $L_q$-space} associated with $M$.
\end{rem}

\subsection{Pisier's method theorem}

Let $M$ be a semifinite von Neumann algebra.

\begin{mdframed}[backgroundcolor=black!10,rightline=false,leftline=false,topline=false,bottomline=false,skipabove=10pt]
A projection $P$ on $L_2(M)$ is said to satisfy \textit{Pisier's method assumptions} if, for every semifinite von Neumann algebra $N$, the amplified projection $Q:=P\otimes I$ on the Hilbertian tensor product $L_2(M)\otimes L_2(N)=L_2(M\bar{\otimes}N)$ is of Jones-type, and the algebraic tensor product $H_\infty(P^{\bot})\odot L_\infty(N)$ is a subspace of $H_\infty(Q^{\bot})$.
\end{mdframed}
\vspace{5pt}

Now we turn to the statement of the main result of the section.

\begin{mdframed}[skipabove=10pt]
\begin{maintheo}[Pisier's method]\label{theorem_PisierMethod}
Let $P$ be a projection on $L_2(M)$ that satifies Pisier's method assumptions. If $0<\theta<1$ and $1<p<\infty$ then
\[[H_p(P),H_1(P)]_\theta=H_{p_\theta}(P)\]
with equivalent norms, with constants depending only on $p,\theta$, where $\frac{1}{p_\theta}=\frac{1-\theta}{p}+\theta$. 
\end{maintheo}
\end{mdframed}
\vspace{5pt}

The proof of Theorem \ref{theorem_PisierMethod} follows Pisier's original approach. Let $P$ be a projection on $L_2(M)$ that satifies Pisier's method assumptions with constant $C$. Let $N$ denote the commutative semifinite von Neumann algebra introduced in the previous paragraph. Let $\tau,\sigma$ denote the trace of $M,N$ respectively. Let $\tau\otimes\sigma$ denote the tensor product trace on the tensor product von Neumann algebra $M\bar{\otimes}N$. Let $Q:=P\otimes I$ denote the amplified projection on the Hilbertian tensor product $L_2(M)\otimes L_2(N)=L_2(M\bar{\otimes}N)$. 

Let $1<p<\infty$ be a fixed exponent, and let $1<q<\infty$ denote the dual exponent, i.e. such that 
\[\frac{1}{p}+\frac{1}{q}=1.\] 
For $z\in\overline{\S}$, let $p_z\in[1,p]$ and $q_z\in[q,\infty]$ be the exponents such that
\[\frac{1}{p_z}=\frac{1-\re(z)}{p}+\re(z),\ \ \ \ \ \ \ \ \ \ \ \frac{1}{q_z}=\frac{1-\re(z)}{q}.\]  
Note that we have 
\[\frac{1}{p_z}+\frac{1}{q_z}=1.\]
Moreover, we have $\re(z)<1$ if and only if $p_z\neq1$ and $q_z\neq\infty$.

As $P$ satisfies the technical assumptions, the compatible couple $(H_p(P),H_1(P))$ is regular, so that we can consider the dual compatible couple $(H_{p}(P)^*,H_{1}(P)^*)$ as defined in the preliminary section of the paper. Then, there is a contractive compatible (antilinear) operator 
\[I^*:(L_q(M),L_\infty(M))\to(H_p(P)^*,H_1(P)^*)\]
such that
\[\bra x,I^*(y)\ket=\tau(xy^*),\ \ \ \text{for}\ y\in L_q(M)+L_\infty(M),\ x\in H_p(P)\cap H_1(P).\]
By the Hahn-Banach theorem, if $z\in\overline{\S}$ then $I^*$ induces a contractive and surjective (antilinear) operator 
\[I^*:L_{q_z}(M)\to H_{p_z}(P)^*\]

\begin{lemm}\label{step1_lemm0}
The operator $I^*:L_q(M)\cap L_\infty(M)\to H_p(P)^*\cap H_1(P)^*$ is surjective.
\end{lemm}
\begin{proof}
As the subcouple $(H_q(P^{\bot}),H_\infty(P^{\bot}))$ is quasi-complemented in $(L_q(M),L_\infty(M))$, we can consider the quotient compatible couple $(L_q(M)/H_q(P^{\bot}),L_\infty(M)/H_\infty(P^{\bot})$. As $(L_q(M)/H_q(P^{\bot}),L_\infty(M)/H_\infty(P^{\bot}))$ and $(H_p(P)^*,H_1(P)^*)$ are canonicaly (anti-)isomorphic as compatible couples, we are reduced to justifying the surjectivity of the projection $L_q(M)\cap L_\infty(M)\to L_q(M)/H_q(P^{\bot})\cap L_\infty(M)/H_\infty(P^{\bot})$. This is proved in Proposition \ref{preliminaries_quotient_couple}.
\end{proof}

\begin{lemm}
If $z\in\overline{\S}$, then the kernel of the operator $I^*:L_{q_z}(M)\to H_{p_z}(P)^*$ is $H_{q_z}(P^{\bot})$.
\end{lemm}
\begin{proof}
By definition the kernel of this operator is
\[\big\{y\in L_{q_z}(M)\ :\ \tau(xy^*)=0,\ \forall x\in H_p(P)\cap H_1(P)\big\}.\]
As $H_p(P)\cap H_1(P)$ is a norm-dense subspace of $H_{p_z}(P)$, we see that the above set coincides with the orthogonal of $H_{p_z}(P)$ w.r.t. the duality bewteen $L_{q_z}(M)$ and $L_{p_z}(M)$, i.e. $H_{q_z}(P^{\bot})$ because $P$ satisfies the technical assumptions.
\end{proof}

As $P$ is of Jones-type, and because $P$ is self-adjoint, we easily get the following result.

\begin{lemm}
If $z\in\overline{\S}$, $\re(z)<1$, then $P$ extends to a bounded and idempotent compatible operator
\[P:(L_{p_z}(M),L_{q_z}(M))\to(L_{p_z}(M),L_{q_z}(M))\]
such that the two operators $P:L_{p_z}(M)\to L_{p_z}(M)$ and $P:L_{q_z}(M)\to L_{q_z}(M)$ have range $H_{p_z}(P)$ and $H_{q_z}(P)$ respectively, and have kernel $H_{p_z}(P^{\bot})$ and $H_{q_z}(P^{\bot})$ respectively. Finally, the operators 
\[P:L_{p_z}(M)\to L_{p_z}(M),\ \ \ P:L_{q_z}(M)\to L_{q_z}(M)\]
are conjugate to each others w.r.t. trace duality between $L_{p_z}(M)$ and $L_{q_z}(M)$ in the sense that 
\[\tau(P(x)y^*)=\tau(xP(y)^*)\]
for $x\in L_{p_z}(M)$, $y\in L_{q_z}(M)$.
\end{lemm}

\begin{lemm}\label{step1_lemm2}
If $z\in\overline{\S}$, $\re(z)<1$ and $y\in L_{q_z}(M)$, then $I^*(y)=I^*(P(y))$ and 
\[\|P(y)\|_{L_{q_z}(M)}\pp\|P\|_{L_{p_z}(M)\to L_{p_z}(M)}\|I^*(y)\|_{H_{p_z}(P)^*}.\]
\end{lemm}
\begin{proof}
As the two operators $P:L_{q_z}(M)\to L_{q_z}(M)$ and $I^*:L_{q_z}(M)\to H_{p_z}(P)^*$ have the same kernel, namely $H_{q_z}(P^{\bot})$, we have $I^*(y)=I^*(P(y))$ for every $y\in L_{q_z}(M)$. Besides, if $y\in L_{q_z}(M)$, we have
\begin{align*}
\|P(y)\|_{L_{q_z}(M)}&=\|P(y)\|_{L_{p_z}^{\times}(M)}\\
&=\sup_{x\in (L_p\cap L_1)(M),\ \|x\|_{L_{p_z}(M)}\pp1}|\tau(xP(y)^*)|\\
&=\sup_{x\in (L_p\cap L_1)(M),\ \|x\|_{L_{p_z}(M)}\pp1}|\tau(P(x)y^*)|\\
&=\sup_{x\in (L_p\cap L_1)(M),\ \|x\|_{L_{p_z}(M)}\pp1}|\bra P(x),I^*(y)\ket|\\
&\pp\sup_{x\in (L_p\cap L_1)(M),\ \|x\|_{L_{p_z}(M)}\pp1}\|I^*(y)\|_{H_{p_z}(P)^*}\|P(x)\|_{H_{p_z}(P)}\\
&=\|I^*(y)\|_{H_{p_z}(P)^*}\|P\|_{L_{p_z}(M)\to L_{p_z}(M)}.
\end{align*}
\end{proof} 

As $Q=P\otimes I$ satisfies the same assumptions as $P$, the compatible couple $(H_{p,1}(Q),H_1(Q))$ is also regular, so that we can consider the dual compatible couple $(H_{p,1}(Q)^*,H_{1}(Q)^*)$. Then, as before, there is a contractive compatible (antilinear) operator 
\[J^*:(L_{q,\infty}(M\bar{\otimes}N),L_{\infty}(M\bar{\otimes}N))\to(H_{p,1}(Q)^*,H_{1}(Q)^*)\] such that
\[\bra x,J^*(y)\ket=(\tau\otimes\sigma)(xy^*),\ \ \ \text{for}\ y\in(L_{q,\infty}+L_\infty)(M\bar{\otimes}N),\ x\in H_{p,1}(Q)\cap H_1(Q),.\]
Again, by the Hahn-Banach theorem, if $z\in\overline{\S}$ then $J^*$ induces a contractive and surjective (antilinear) operator 
\[J^*:L_{q_z,\infty}(M\bar{\otimes}N)\to H_{p_z,1}(Q)^*\]
The proofs of the three following results are similar to the previous ones.

\begin{lemm}
If $z\in\overline{\S}$, then the kernel of the operator $J^*:L_{q_z,\infty}(M\bar{\otimes}N)\to H_{p_z,1}(Q)^*$ is $H_{q_z,\infty}(Q^{\bot})$.
\end{lemm}

\begin{lemm}
If $z\in\overline{\S}$, $\re(z)<1$, then $Q$ extends to a compatible bounded operator
\[Q:(L_{p_z,1}(M\bar{\otimes}N),L_{q_z,\infty}(M\bar{\otimes}N))\to(L_{p_z,1}(M\bar{\otimes}N),L_{q_z,\infty}(M\bar{\otimes}N))\]
such that the two operators $Q:L_{p_z,1}(M\bar{\otimes}N)\to L_{p_z,1}(M\bar{\otimes}N)$ and $Q:L_{q_z,\infty}(M\bar{\otimes}N)\to L_{q_z,\infty}(M\bar{\otimes}N)$ have range $H_{p_z,1}(Q)$ and $H_{q_z,\infty}(Q)$ respectively, and have kernel $H_{p_z,1}(Q^{\bot})$ and $H_{q_z,\infty}(Q^{\bot})$ respectively. Finally, the operators \[Q:L_{p_z,1}(M\bar{\otimes}N)\to L_{p_z,1}(M\bar{\otimes}N),\ \ \ Q:L_{q_z,\infty}(M\bar{\otimes}N)\to L_{q_z,\infty}(M\bar{\otimes}N)\]
are conjugate to each other w.r.t. trace duality betwenn $L_{p_z,1}(M\bar{\otimes}N)$ and $L_{q_z,\infty}(M\bar{\otimes}N)$.
\end{lemm}

\begin{lemm}\label{step1_lemm3}
If $z\in\overline{\B}$, $\re(z)<1$ and $y\in L_{q_z,\infty}(M\bar{\otimes}N)$, then $J^*(y)=J^*(Q(y))$ and 
\[\|Q(y)\|_{L_{q_z,\infty}(M\bar{\otimes}N)}\pp\|Q\|_{L_{p_z,1}(M\bar{\otimes}N)\to L_{p_z,1}(M\bar{\otimes}N)}\|J^*(y)\|_{H_{p_z,1}(Q)^*}.\]
\end{lemm}

For $z\in\overline{\S}$, let 
\[\fonct{F_z}{L_{q_z}(M)}{L_{q_z,\infty}(M\bar{\otimes}N)}{x}{x\otimes f_z}\]
denote the amplification operator studied in the previous section. 

\begin{lemm}\label{step1_lemm4}
If $z\in\overline{\S}$, $\re(z)<1$, we have $Q(F_z(x))=F_z(P(x))$ for every $x\in L_{q_z}(M)$.
\end{lemm}
\begin{proof}
By definition, we have $Q(x\otimes f)=P(x)\otimes f$ for $x\in(L_2\cap L_{q_z})(M)$ and $f\in(L_2\cap L_{q_z,\infty})(N)$, and thus this extends for every $x\in L_{q_z}(M)$ and $f\in L_{q_z,\infty}(N)$ because $P:L_{q_z}(M)\to L_{q_z}(M)$ and $Q:L_{q_z,\infty}(M\bar{\otimes}N)\to L_{q_z,\infty}(M\bar{\otimes}N)$ are w*\-/continuous as they are the antiduals of $P:L_{p_z}(M)\to L_{p_z}(M)$ and $Q:L_{p_z,1}(M\bar{\otimes}N)\to L_{p_z,1}(M\bar{\otimes}N)$ respectively. In particular, if $x\in L_{q_z}(M)$, we get
\[Q(F_z(x))=Q(x\otimes f_z)=P(x)\otimes f_z=F_z(P(x))\]
as desired.
\end{proof}

\begin{lemm}
If $z\in\overline{\S}$, then $F_z:L_{q_z}(M)\to L_{q_z,\infty}(M\bar{\otimes}N)$ maps $H_{q_z}(P^{\bot})$ into $H_{q_z,\infty}(Q^{\bot})$. 
\end{lemm}
\begin{proof}
By Pisier's method assumptions, we know that $H_\infty(P^{\bot})\odot L_\infty(M)$ is a subspace of $H_\infty(Q^{\bot})$. As a consequence, if $\re(z)=1$, then $F_z$ clearly maps $H_\infty(P^{\bot})$ into $H_\infty(Q^{\bot})$. Now, we assume $\re(z)<1$. If $x\in H_{q_z}(P^{\bot})$, then $P(x)=0$, and thus by Lemma \label{step1_lemma4} we get $Q(F_z(x))=F_z(P(x))=0$, showing that $F_z(x)\in H_{q_z,\infty}(Q^{\bot})$ as required.
\end{proof}

As a consequence of the above lemma, for $z\in\overline{\S}$ there is a unique operator 
\[T_z:H_{p_z}(P)^*\to H_{p_z,1}(Q)^*\]
that makes the following diagram commute.
\begin{center}
\begin{tikzcd}
L_{q_z}(M) \arrow[r, "F_z"] \arrow[d, "I^*"']       & {L_{q_z,\infty}(M\bar{\otimes}N)} \arrow[d, "J^*"]   \\
H_{p_z}(P)^* \arrow[r, "T_z"] & {H_{p_z,1}(Q)^*}
\end{tikzcd}
\end{center}

\begin{lemm}\label{step1_lemm5}
The family of operators $(T_z)_{z\in\overline{\S}}$ satisfies the following properties.
\begin{itemize}  
    \item if $\phi\in H_p(P)^*\cap H_1(P)^*$, the function 
\[\fonctbis{\overline{\S}}{H_{p,1}(Q)^*+H_1(Q)^*}{z}{T_{\overline{z}}(\phi)}\]
belongs to $\mathcal{F}_*(H_{p,1}(Q)^*,H_1(Q)^*)$,
    \item for every $z\in\overline{\S}$ and $\phi\in H_{p_z}(P)^*$ we have
\[\|T_z(\phi)\|_{H_{p_z,1}(Q)^*}\pp p_z\|\phi\|_{H_{p_z}(P)^*}\]
and if in addition $\re(z)<1$, then
\[\|\phi\|_{H_{p_z}(P)^*}\pp p_z^{-1}\|Q\|_{L_{p_z,1}(M\bar{\otimes}N)\to L_{p_z,1}(M\bar{\otimes}N)}\|T_z(\phi)\|_{H_{p_z,1}(Q)^*}\]
\end{itemize}
\end{lemm}
\begin{proof}
The proof of the lemma is based on Theorem \ref{amplification_theo1}.

$\bullet$ Fix $\phi\in H_p(P)^*\cap H_1(P)^*$. Then by Lemma \ref{step1_lemm0} we can find $y\in(L_q\cap L_\infty)(M)$ such that $\phi=I^*(y)$. If $z\in\overline{\S}$, then by definition of $T_z$ we have $T_z(\phi)=J^*(F_z(y))$. As the antilinear operators
\[J^*:L_{q,\infty}(M\bar{\otimes}N)\to H_{p,1}(M\bar{\otimes}N)^*\ \ \ \ \ \ \ \ \ J^*:L_\infty(M\bar{\otimes}N)\to H_1(M\bar{\otimes}N)^*\]
are clearly *-weakly continuous as dual operators, and because we know that the function 
\[\fonctbis{\overline{\S}}{L_{q,\infty}(M\bar{\otimes}N)+L_\infty(M\bar{\otimes}N)}{z}{F_z(y)}\]
belongs to $\mathcal{F}_*(L_{q,\infty}(M\bar{\otimes}N),L_{q,\infty}(M\bar{\otimes}N))$, by an application of Proposition \ref{steintheorem_prop1} we get the desired conclusion.

$\bullet$ Fix $z\in\overline{\S}$ and $\phi\in H_{p_z}(P)^*$. Then, by the Hahn-Banach theorem there is $y\in L_{q_z}(M)$ such that $\phi=I^*(y)$ and $\|y\|_{L_{q_z}(M)}=\|\phi\|_{H_{p_z}(P)^*}$. Then
\begin{align*}
\|T_z(\phi)\|_{H_{p_z,1}(Q)^*}&=\|J^*(F_z(y))\|_{H_{p_z,1}(Q)^*}\\
&\pp \|F_z(y)\|_{L_{q_z,\infty}(M\bar{\otimes}N)}\\
&=p_z\|y\|_{L_{q_z}(M)}\\
&=p_z\|\phi\|_{H_{p_z}(P)^*}
\end{align*}
Now we assume that $\re(z)<1$. By Lemma \ref{step1_lemm2} we have $\phi=I^*(y)=I^*(P(y))$, and thus
\[\|\phi\|_{H_{p_z}(P)^*}\pp\|P(y)\|_{L_{q_z}(M)}.\]
By Proposition \ref{step1_lemm4}, we get
\begin{align*}
\|P(y)\|_{L_{q_z}(M)}&=p_z^{-1}\|F_z(P(y))\|_{L_{q_z,\infty}(M\bar{\otimes}N)}\\
&=p_z^{-1}\|Q(F_z(y))\|_{L_{q_z,\infty}(M\bar{\otimes}N)}.\\
\end{align*}
Moreover, by Lemma \ref{step1_lemm3} we have
\begin{align*}
\|Q(F_z(y))\|_{L_{q_z,\infty}(M\bar{\otimes}N)}&\pp\|Q\|_{L_{p_z,1}(M\bar{\otimes}N)\to L_{p_z,1}(M\bar{\otimes}N)}\|J^*(F_z(y))\|_{H_{p_z,1}(Q)^*}\\
&=\|Q\|_{L_{p_z,1}(M\bar{\otimes}N)\to L_{p_z,1}(M\bar{\otimes}N)}\|T_z(I^*(y))\|_{H_{p_z,1}(Q)^*}\\
&=\|Q\|_{L_{p_z,1}(M\bar{\otimes}N)\to L_{p_z,1}(M\bar{\otimes}N)}\|T_z(\phi)\|_{H_{p_z,1}(Q)^*}
\end{align*}
By combining the above estimates we get the desired result.
\end{proof}

\begin{lemm}\label{step1_lemm1}
If $0<\theta<1$, we have 
\[H_{p_\theta,1}(Q)^*=(H_{p,1}(Q)^*,H_1(Q)^*)_{\theta,\infty,K}\]
with equivalent norms, with constants depending only on $p,\theta$.
\end{lemm}
\begin{proof}
As $Q$, and thus also $Q^{\bot}$, is of Jones-type, we know that $(H_1(Q^{\bot}),H_\infty(Q^{\bot}))$ is quasi-complemented in $(L_1(M\bar{\otimes}N),L_\infty(M\bar{\otimes}N))$ with a universal constant. By Theorem \ref{preliminaries_Holmstedt}, we deduce that $(H_{q,\infty}(Q^{\bot}),H_\infty(Q^{\bot}))$ is quasi-complemented in $(L_{q,\infty}(M\bar{\otimes}N),L_\infty(M\bar{\otimes}N))$ with a constant depending only on $p$. Moreover, by the reiteration theorem for the real method together with the equivalence theorem between the $K$-method and the $J$-method, we have
\[L_{q_\theta,\infty}(M\bar{\otimes}N)=(L_{q,\infty}(M\bar{\otimes}N),L_\infty(M\bar{\otimes}N))_{\theta,\infty,J}\]
with equivalent norms, with constants depending only on $p,\theta$. As a consequence of Proposition \ref{technical_assumptions_Jclosedness}, we get
\[L_{q_\theta,\infty}(M\bar{\otimes}N)/H_{q_\theta,\infty}(Q^{\bot})=(L_{q,\infty}(M\bar{\otimes}N)/H_{q,\infty}(Q^{\bot}),L_\infty(M\bar{\otimes}N)/H_\infty(Q^{\bot}))_{\theta,\infty,K}\]
with equivalent norms, with constants depending only on $p,\theta$. The desired conclusion follows.
\end{proof}





Finally, we are now able to complete the proof of Theorem \ref{theorem_PisierMethod}.

Fix $\phi\in H_p(P)^*\cap H_1(P)^*$ and $0<\theta<1$. Then, we can apply our weak-* version of Stein's interpolation theorem to the analytic family of operators $(p^{z-1}T_{\overline{z}})_{z\in\overline{\S}}$. It follows that $T_\theta(\phi)\in(H_{p,1}(Q)^*,H_1(Q)^*)_{\theta,\infty,K}$ with 
\[\|T_\theta(\phi)\|_{(H_{p,1}(Q)^*,H_1(Q)^*)_{\theta,\infty,K}}\pp p^{1-\theta}\|\phi\|_{[H_p(P)^*,H_1(P)^*]_\theta}\]
By Lemma \ref{step1_lemm1} we deduce that $T_\theta(\phi)\in H_{p_\theta,1}(Q)^*$ with
\[\|T_\theta(\phi)\|_{H_{p_\theta,1}(Q)^*}\pp C_{p,\theta}p^{1-\theta}\|\phi\|_{[H_p(P)^*,H_1(P)^*]_\theta}\]
where $C_{p,\theta}$ depends on $p,\theta$ only. Finally, by the last point of Theorem \ref{step1_lemm5}, we deduce that $\phi\in H_{p_\theta}(P)^*$ with
\[\|\phi\|_{H_{p_\theta}(P)^*}\pp C_{p,\theta} p^{1-\theta}p_\theta^{-1}\|Q\|_{L_{p_\theta,1}(M\bar{\otimes}N)\to L_{p_\theta,1}(M\bar{\otimes}N)}\|\phi\|_{[H_p(P)^*,H_1(P)^*]_\theta}\]
By the duality theorem for the complex method (more precisely Corollary \ref{ComplexInterpolation_Dual}), we deduce that, for every $x\in H_{p_\theta}(P)$, we have $x\in[H_p(P),H_1(P)]_\theta$, with
\[\|x\|_{[H_p(P),H_1(P)]_\theta}\pp C_{p,\theta}p^{1-\theta}p_\theta^{-1}\|Q\|_{L_{p_\theta,1}(M\bar{\otimes}N)\to L_{p_\theta,1}(M\bar{\otimes}N)}\|x\|_{H_{p_\theta}(M)}.\]
The proof is complete.

\begin{rem}
A careful application of Holmstedt's formula in Lemma \ref{step1_lemm1} yields the estimate
\[C_{p,\theta}\lesssim [\theta(1-\theta)]^{-1}\]
Moreover, under the additional assumption that $Q$ satisfies a weak-type $(1,1)$ inequality with a universal constant, Marcinkiewicz interpolation theorem gives
\[\|Q\|_{L_{p_\theta,1}(M\bar{\otimes}N)\to L_{p_\theta,1}(M\bar{\otimes}N)}\lesssim p_\theta q_\theta.\]
As a consequence, for $x\in H_{p_\theta}(P)$ we obtain the estimate
\[\|x\|_{[H_p(P),H_1(P)]_\theta}\lesssim  [\theta(1-\theta)]^{-1}p^{1-\theta}q_\theta\|x\|_{H_{p_\theta}(M)}.\]
This is the best estimate that can be obtained by this method.
\end{rem}

\section{Exemple: Analytic Hardy spaces}

In this paragraph, we detail how the well-known case of analytic Hardy spaces fit into our framework of abstract Hardy spaces.

\vspace{10pt}

Let $M$ be a semifinite von Neumann algebra, with trace denoted $\tau$, with modular conjugation denoted $J$, and equipped with a \textit{subdiagonal algebra}, i.e. a w*-closed (but note necessarily $\ast$-closed) subalgebra $A$ of $M$ such that $A+J(A)$ is w*-dense in $M$, the restriction of the trace to the \textit{diagonal} $D:=A\cap J(A)$ is still semifinite, and $E_D(xy)=E_D(x)E_D(y)$ for every $x,y\in A$, where $E_D$ denotes the conditional expectation of $M$ onto $D$. 

If $E(M)$ is an exact interpolation space for the compatible couple $(L_1(M),L_\infty(M))$, we define $H_E(A)$ to be the $\sigma(E(M),E^{\times}(M))$\-/closure of $A\cap E(M)$ in $E(M)$. We will also denote $H_p(A):=H_{L_p}(A)$ for $1\pp p\pp\infty$. They are the so-called \textit{Hardy spaces} associated with $A$, as defined in \cite{BekjanSubdiagonalHardySpaces}. 

\begin{rem}
if $E(M)$ is an exact interpolation space for $(L_1(M),L_\infty(M))$, then for $x\in H_E(A)$ and $y\in H_{E^{\times}}(A)$ we clearly have
\begin{equation}\label{subdiagonal_multiplicativity}
E_D(xy)=E_D(x)E_D(y)
\end{equation}
\end{rem}

Let denote
\[A^0:=\{x\in A\ :\ E_D(x)=0\}.\]
If $E(M)$ is an exact interpolation space for $(L_1(M),L_\infty(M))$, we define $H_E^0(A)$ to be the $\sigma(E(M),E^{\times}(M))$-weak closure of $A^0\cap E(M)$ in $E(M)$. Then, clearly 
\[H_E^0(A)=\big\{x\in H_E(M)\ :\ E_D(x)=0\big\}.\]
We will also denote $H_p^0(A):=H_{L_p}^0(A)$ for $1\pp p\pp\infty$.

\begin{lemm}\label{Saito_Lemma}
If $M$ is finite, then 
\[H_1(A)=\{x\in L_1(M)\ :\ \tau(xy)=0,\ \forall y\in A^0\}\]
and 
\[H_1^0(A)=\{x\in L_1(M)\ :\ \tau(xy)=0,\ \forall y\in A\}.\]
\end{lemm}
\begin{proof}
This is proved for instance in \cite{SaitoSubdiagonalHardySpaces}[Lemma 3, Lemma 4].
\end{proof}

The proof of the following lemma is straightforward.

\begin{lemm}
Let $p\in D$ be a finite projection. Then $pAp$ is a subdiagonal algebra of $pMp$ with diagonal $pDp$. Moreover, if $E(M)$ is an exact interpolation space for $(L_1(M),L_\infty(M))$, then $H_{pEp}(pAp)=pH_E(A)p$ and $H_{pEp}^0(pAp)=pH_E^0(A)p$.
\end{lemm}

\begin{mdframed}[skipabove=10pt]
\begin{theo}\label{subdiagonal_theo1}
Let $E(M)$ be an exact interpolation space for $(L_1(M),L_\infty(M))$ with order\-/continuous norm. Then 
\[H_E(A)=\{x\in E(M)\ :\ \tau(xy)=0,\ \forall y\in H_{E^{\times}}^0(A)\}\]
and 
\[H_E^0(A)=\{x\in E(M)\ :\ \tau(xy)=0,\ \forall y\in H_{E^\times}(A)\}.\]
\end{theo}
\end{mdframed}
\begin{proof}
If $x\in H_E(A)$ and $y\in H_{E^{\times}}^0(A)$, then by \eqref{subdiagonal_multiplicativity} we have
\[\tau(xy)=\tau(E_D(xy))=\tau(E_D(x)E_D(y))=\tau(0)=0.\]
Hence we obtaine one inclusion. By contradiction, we assume the converse inclusion is strict, so that there is $x\in E(M)\setminus H_E(A)$ such that $\tau(xy)=0$ for every $y\in H_{E^{\times}}^0(A)$. As $E(M)$ has order-continuous norm, by the Hahn-Banach theorem there is $y\in E^{\times}(M)$ such that $\tau(xy)\neq0$ and $\tau(x'y)=0$ for every $x\in H_E(A)$. Since the restriction of the trace to $D$ is semifinite, there is an increasing net $(p_\alpha)_\alpha$ of finite projections of $D$ such that $\sup_\alpha p_\alpha=1$. Then, by Lemma \ref{Saito_Lemma}, we have $p_\alpha xp_\alpha\in H_1(p_\alpha Ap_\alpha)$ and $p_\alpha y^*p_\alpha\in H_1^0(p_\alpha Ap_\alpha)$. Indeed, if $x_\alpha'\in p_\alpha Ap_\alpha$, then $x_\alpha'\in H_E(A)$ and thus
\[\tau(x_\alpha'p_\alpha yp_\alpha)=\tau(p_\alpha x'_\alpha p_\alpha y)=\tau(x_\alpha' y)=0.\]
Similarly, if $y_\alpha'\in (p_\alpha Ap_\alpha)^0$, then $y_\alpha'\in H_{E^{\times}}^0(A)$ and thus
\[\tau(p_\alpha xp_\alpha y'_\alpha)=\tau(xp_\alpha y_\alpha'p_\alpha)=\tau(xy'_\alpha)=0.\]
Thus, if $\gamma$ is such that $\gamma\pg\alpha$ and $\gamma\pg\beta$, we have $p_\alpha xp_\alpha\in H_1(p_\gamma Ap_\gamma)$ and $p_\beta yp_\beta\in J(H_1^0(p_\gamma Ap_\gamma))$, and by \eqref{subdiagonal_multiplicativity} we deduce that
\[E_{p_\gamma Dp_\gamma}(p_\alpha xp_\alpha p_\beta yp_\beta)=E_{p_\gamma Dp_\gamma}(p_\alpha xp_\alpha)E_{p_\gamma Dp_\gamma}(p_\beta yp_\beta)=0.\]
In particular, we get
\[\tau(p_\alpha xp_\alpha p_\beta yp_\beta)=\tau(E_{p_\gamma Dp_\gamma}(p_\alpha xp_\alpha p_\beta yp_\beta))=0.\]
As the net $(p_\alpha xp_\alpha)_\alpha$ converges to $x$ in $E(M)$ for the norm topology and the net $(p_\alpha yp_\alpha)_\alpha$ converges to $y$ in $E^{\times}(M)$ for the w*-topology, we deduce that $\tau(xy)=0$, hence the contradiction. Hence $H_{E^{\times}}^0(A)$ is the orthogonal of $H_E(A)$. Similarly, we prove that $H_{E^{\times}}(A)$ is the orthogonal of $H_E^0(A)$.
\end{proof}

\begin{mdframed}[skipabove=10pt]
\begin{theo}\label{subdiagonal_theo2}
Let $E(M)$ be an exact interpolation space for $(L_1(M),L_\infty(M))$ with order\-/continuous norm. Then $A\cap L_1(M)$ reps. $A^0\cap L_1(M)$ is dense in $H_E(A)$ resp. $H_E^0(A)$ for the norm topology of $E(M)$ and is dense in $H_{E^{\times}}(A)$ resp. $H_{E^{\times}}^0(A)$ for the \text{\upshape w*}-topology of $E^{\times}(M)$.
\end{theo}
\end{mdframed}
\begin{proof}
It suffices to prove the statement for $H_E(A)$ and $H_{E^{\times}}(A)$. Since the restriction of the trace to $D$ is semifinite, there is an increasing net $(p_\alpha)_\alpha$ of finite projections of $D$ such that $\sup_\alpha p_\alpha=1$. If $x\in A\cap E(M)$ resp. $x\in A\cap E^{\times}(M)$ then we know that $(p_\alpha xp_\alpha)_\alpha$ converges to $x$ for the norm topology of $E(M)$ resp. for the w*-topology of $E^{\times}(M)$. As $p_\alpha xp_\alpha$ clearly belongs to $A\cap L_1(M)$, the proof is complete. 
\end{proof}

The associated \textit{Riesz projection} $P$ is the projection of $L_2(M)$ onto $H_2(A)$.

\begin{mdframed}[skipabove=10pt]
\begin{theo}\label{subdiagonal_theo3}
Let $E(M)$ be an exact interpolation space for $(L_1(M),L_\infty(M))$ with order\-/continuous norm. Then
\[H_E(P)=H_E(A),\ \ \ \ \ H_{E^{\times}}(P)=H_{E^{\times}}(A),\]
\[H_E(P^{\bot})=J(H_E^0(A)),\ \ \ \ \ H_{E^{\times}}(P^{\bot})=J(H_{E^{\times}}^0(A)).\]
\end{theo}
\end{mdframed}
\begin{proof}
From Theorem \ref{subdiagonal_theo1}, it is clear that $H_2(M)\cap E(M)$ sits between $A\cap L_1(M)$ and $A\cap E(M)$. Thus, from Theorem \ref{subdiagonal_theo2} we deduce that $H_E(P)=H_E(A)$. Besides, by Theorem \ref{subdiagonal_theo2} the range of $P^{\bot}$ is $J(H_2^0(A))$, so that $H_E(P)$ is the norm-closure of $J(H_2^0(A))\cap E(M)$ in $E(M)$. Again, as a consequence of Theorem \ref{subdiagonal_theo1} and Theorem \ref{subdiagonal_theo2}, we see that the latter coincides with $J(H_E^0(A))$.
\end{proof}

\begin{mdframed}[skipabove=10pt]
\begin{theo}\label{subdiagonal_theo4}
The Riesz projection $P$ is of Jones-type.
\end{theo}
\end{mdframed}
\begin{proof}
The previous results (together with the bipolar theorem) show that $P$ satisfies the technical assumptions. It is also proved in \cite{BekjanSubdiagonalHardySpaces}[Lemma 6.2] that the subcouple $(H_1(A),H_2(A))$ is $K$-closed in $(L_1(M),L_2(M))$ with a universal constant. As a consequence, the subcouple $(H_1^0(A),H_2^0(A))$ is also $K$-closed in $(L_1(M),L_2(M))$ with a universal constant. By Theorem \ref{subdiagonal_theo3}, this shows that $P$ is indeed of Jones-type.
\end{proof}

\begin{lemm}\label{subdiagonal_lemma}
Let $N$ be a semifinite von Neumann algebra, and let $A\bar{\otimes}N$ denote the \text{\upshape w*}\-/closure of $A\odot N$ in $M\bar{\otimes}N$. Then $A\bar{\otimes}N$ is a subdiagonal algebra of $M\bar{\otimes}N$. Moreover, the amplified projection $P\otimes I$ on the hilbertian tensor product $L_2(M\bar{\otimes}N)$ coincides with the Riesz projection associated with $A\bar{\otimes}N$.
\end{lemm}
\begin{proof}
It is routine to check that $A\bar{\otimes}N$ is indeed a subdiagonal algebra of $M\bar{\otimes}N$. Let $Q$ denote the Riesz projection associated with $A\bar{\otimes}N$. Fix $u\in A\cap L_2(M)$, $v\in A^0\cap L_2(M)$ and $y\in N\cap L_2(N)$. Then we have $u\otimes y\in A\bar{\otimes}N\cap L_2(M\bar{\otimes}N)\subset H_2(A\bar{\otimes}N)$ and $v\otimes y^*\in (A\bar{\otimes}N)^0\cap L_2(M\bar{\otimes}N)\subset H_2^0(A\bar{\otimes}N)$. By Theorem \ref{subdiagonal_theo1} we know that $P,Q$ are the projections with range $H_2(A),H_2(A\bar{\otimes}N)$ and kernel $J(H_2^0(A)),J(H_2^0(A\bar{\otimes}N))$, thus we get
\[Q((u+v^*)\otimes y)=Q(u\otimes y)+Q((v\otimes y^*)^*)=u\otimes y=P(u)\otimes y\]
As $A\cap L_2(M)$ is dense in $H_2(A)$, $A^0\cap L_2(M)$ is dense in $H_2^0(A)$ and $N\cap L_2(N)$ is dense in $L_2(N)$, we deduce that 
\[Q(x\otimes y)=P(x)\otimes y\]
for every $x\in H_2(A)+J(H_2^0(A))$ and $y\in L_2(M)$. Thus $Q=P\otimes I$ as desired.
\end{proof}

\begin{mdframed}[skipabove=10pt]
\begin{theo}\label{subdiagonal_theo5}
The Riesz projection $P$ satisfies Pisier's method assumptions.
\end{theo}
\end{mdframed}
\begin{proof}
This is a straightforward consequence of Theorem \ref{subdiagonal_theo4} together with Lemma \ref{subdiagonal_lemma}.
\end{proof}

By applying Theorem \ref{theorem_PisierMethod}, we obtain that, if $0<\theta<1$ and $1<p<\infty$ then
\[[H_p(A),H_1(A)]_\theta=H_{p_\theta}(A)\]
with equivalent norms, with constants depending only on $p,\theta$, where $\frac{1}{p_\theta}=\frac{1-\theta}{p}+\theta$. Naturally, this partially recovers the results obtained in \cite{BekjanSubdiagonalHardySpaces} where the main result treats the general case $0<p\pp\infty$. As mentionned in the introduction, by our approach we cannot reach such a level of generality since we restrict ourselves within the framework of normed spaces.

\clearpage

\part{Applications to martingale theory}

\section{Filtrations}

Let $M$ be a semifinite von Neumann algebra, with trace denoted $\tau$, and equipped with a \textit{filtration}, i.e. an increasing sequence $(M_n)_{n\pg1}$ of von Neumann subalgebras of $M$ whose union $\cup_{n\pg1}M_n$ is w*\-/dense in $M$ and such that there is a trace\-/preserving normal faithful conditional expectation $E_n$ of $M$ onto $M_n$ for every $n\pg1$. Then $(E_n)_{n\pg1}$ is an increasing sequence of commuting projections. For every $n\pg1$, we set
\[D_n:=E_n-E_{n-1}\]
(with the convention $E_0:=0$). Then $(D_n)_{n\pg1}$ is a sequence of mutually orthogonal projections that commute with the $(E_n)_{n\pg1}$. We will refer to them as the \textit{increment projections} associated with the filtration. 

\begin{theo}\label{filtrations_theorem1}
Let $E(M)$ be an exact interpolation space for $(L_1(M),L_\infty(M))$ with order\-/continuous norm.
\begin{enumerate}[nosep]
    \item If $x\in E(M)$, then the sequence $(E_n(x))_{n\pg1}$ converges to $x$ in $E(M)$ for the norm topology. 
    \item If $y\in E^{\times}(M)$, then the sequence $(E_n(y))_{n\pg1}$ converges to $y$ in $E^{\times}(M)$ for the \text{\upshape w*}\-/topology.
\end{enumerate}
\end{theo}
\begin{proof}
Let $x\in E(M)$ and $\epsilon>0$. As $\cup_{n\pg1}M_n$ is w*-dense in $M$, we deduce that $\cup_{n\pg1}(L_1\cap L_\infty)(M_n)$ is $\sigma(E(M),E^{\times}(M))$-weakly dense in $E(M)$, and thus norm-dense in $E(M)$ because $E(M)$ has an order-continuous norm. Hence, there is $k\pg1$ and $y\in(L_1\cap L_\infty)(M_k)$ such that $\|x-y\|_{E(M)}<\epsilon$. Then, for all $n\pg k$, we have
\begin{align*}
\|E_n(x)-x\|_{E(M)}&=\|E_n(x)+E_n(y)+y-x\|_{E(M)}\\
&\pp\|E_n(x-y)\|_{E(M)}+\|x-y\|_{E(M)}\\
&\pp2\|x-y\|_{E(M)}<2\epsilon
\end{align*}
which shows that $(E_n(x))_{n\pg1}$ converges in norm to $x$. Now, if $y\in E^{\times}(M)$ then for every $x\in E(M)$ we get
\[\tau(xE_n(y))=\tau(E_n(x)y)\underset{n\to\infty}{\to}\tau(xy)\]
as desired.
\end{proof}

\begin{coro}\label{filtrations_corollary1}
Let $E(M)$ be an exact interpolation space for $(L_1(M),L_\infty(M))$ with order\-/continuous norm. If $x\in E(M)$ and $y\in E^{\times}(M)$, then 
\[\tau(xy)=\sum_{n=1}^{\infty}\tau(D_n(x)D_n(y)).\]
\end{coro}

\section{Martingale projections}

Let $M$ be a semifinite von Neumann algebra equipped with a filtration $(M_n)_{n\pg1}$ with associated conditional expectations denoted $(E_n)_{n\pg1}$ and associated increment projections denoted $(D_n)_{n\pg1}$. 

Now, let $I$ be a fixed set of positive integers, and let $P$ be the \textit{martingale projection} associated with the data $(M,(M_n)_{n\pg1},I)$, i.e. the projection on $L_2(M)$ such that, if $x\in L_2(M)$ then
\[P(x)=\sum_{n\in I}D_n(x),\ \ \ \text{in}\ L_2(M).\]
Note that the complement projection $P^{\bot}$ coincides with the martingale projection associated with the data $(M,(M_n)_{n\pg1},I^{\bot})$ where $I^{\bot}$ is the complement of $I$ in the set of positive integers. In the sequel, we denote
\[H(P):=\big\{x\in(L_1+L_\infty)(M)\ :\ \forall n\notin I,\ D_n(x)=0\big\}\]
so that we have
\[P(L_2(M))=L_2(M)\cap H(P).\]

\begin{mdframed}[skipabove=10pt]
\begin{theo}\label{martingale_projection_theorem1}
Let $E(M)$ be an exact interpolation space for $(L_1(M),L_\infty(M))$ with order\-/continuous norm. Then, we have
\[H_E(P)=E(M)\cap H(P),\]
\[H_{E^{\times}}(P)=E^{\times}(M)\cap H(P).\]
Moreover, the subspace
\[(L_1\cap L_\infty)(M)\cap H(P)\]
is norm-dense in $H_E(P)$ and \text{\upshape w*}-dense in $H_{E^{\times}}(P)$.
\end{theo}
\end{mdframed}
\begin{proof}
It is clear that
\[E(M)\cap H(P)=\big\{x\in E(M)\ :\ \forall n\notin I,\ D_n(x)=0\big\}\] 
is a norm-closed subspace of $E(M)$ that contains $E(M)\cap P(L_2(M))$, and moreover $E(M)\cap P(L_2(M))$ clearly contains $(L_1\cap L_\infty)(M)\cap H(P)$. Thus, it suffices to show that $(L_1\cap L_\infty)(M)\cap H(P)$ is norm-dense in $E(M)\cap H(P)$. Let $x\in E(M)\cap H(P)$. As the sequence $(E_n(x))_{n\pg1}$ converges to $x$ in $E(M)$ for the norm topology and also belongs to $E(M)\cap H(P)$, we can assume that there is $n\pg1$ such that $E_n(x)=x$, so that we have
\[x=\sum_{k=1}^{n}D_k(x)=\sum_{k\in I,k\pp n}D_k(x).\]
As $(L_1\cap L_\infty)(M)$ is norm-dense in $E(M)$, there is a net $(y_\alpha)_\alpha$ of $(L_1\cap L_\infty)(M)$ that converges to $x$ in $E(M)$ for the norm topology. We set
\[x_\alpha:=\sum_{k\in I,k\pp n}D_k(y_\alpha).\]
Then $x_\alpha\in(L_1\cap L_\infty)(M)\cap H(P)$. As the net $(x_\alpha)_\alpha$ clearly converges to $x$ in $E(M)$ for the norm topology, the proof is complete. The dual statement is proved similarly.
\end{proof}

\begin{mdframed}[skipabove=10pt]
\begin{theo}\label{martingale_projection_theorem2}
The martingale projection $P$ satisfies the technical assumptions.
\end{theo}
\end{mdframed}
\begin{proof}
Let $E(M)$ be an exact interpolation space for $(L_1(M),L_\infty(M))$ with order\-/continuous norm. It suffices to show that the polar of $H_E(P)$, $H_{E^{\times}}(P)$ w.r.t. trace duality coincides respectively with $H_{E^{\times}}(P^{\bot})$, $H_E(P^{\bot})$. This is an easy consequence of Theorem \ref{martingale_projection_theorem1}.
\end{proof}

\begin{mdframed}[skipabove=10pt]
\begin{theo}\label{martingale_projection_theorem4}
The martingale projection $P$ is of Jones-type.
\end{theo}
\end{mdframed}
\begin{proof}
As a martingale transform, it is well-known (see \cite{NarcisseMartingaleTransforms}[Theorem 3.1]) that $P$ is of weak-type $(1,1)$ with a universal constant, and thus is is $L_p$-bounded with a constant depending only on $p$, for every $1<p<\infty$. Moreover, it is proved in \cite{Moyart2024}[Theorem 2.8] that the subcouple $(H_1(P),H_2(P))$ is $K$-closed in $(L_1(M),L_2(M))$ with a universal constant. As $P^{\bot}$ is also a martingale projection, we deduce that the subcouple $(H_1(P^{\bot}),H_2(P^{\bot}))$ is also $K$-closed in $(L_1(M),L_2(M))$ with a universal constant. Thus $P$ is indeed of Jones-type.
\end{proof}

\begin{lemm}\label{martingale_projection_lemma1}
Let $N$ be a tracial von Neumann algebra. Then the amplified projection $Q:=P\otimes I$ on $L_2(M\bar{\otimes}N)=L_2(M)\otimes L_2(N)$ is the martingale projection associated with the data $(M\bar{\otimes}N,(M_n\bar{\otimes}N)_{n\pg1},I)$.
\end{lemm}

\begin{mdframed}[skipabove=10pt]
\begin{theo}\label{martingale_projection_theorem5}
The martingale projection $P$ satisfies Pisier's method assumptions.
\end{theo}
\end{mdframed}
\begin{proof} 
This is a straightforward consequence of Lemma \ref{martingale_projection_lemma1} together with Theorem \ref{martingale_projection_theorem1} and Theorem \ref{martingale_projection_theorem4}.
\end{proof}

\begin{mdframed}[skipabove=10pt]
\begin{coro}\label{martingale_projection_corollary1}
If $1<p<\infty$ and $0<\theta<1$, then
\[[H_p(P),H_1(P)]_\theta=H_{p_\theta}(P)\]
with equivalent norms, with constants depending only on $p,\theta$, where 
\[\frac{1}{p_\theta}=\frac{1-\theta}{p}+\theta.\]
\end{coro}
\end{mdframed}
\vspace{5pt}

\section{Bimartingale projections}

Let $M$ be a semifinite von Neumann algebra equipped with two filtrations $(M_n^{-})_{n\pg1}$, $(M_n^{+})_{n\pg1}$ with associated conditional expectations denoted by $(E_n^{-})_{n\pg1}$, $(E_n^{+})_{n\pg1}$ and increment projections denoted by $(D_n^{-})_{n\pg1}$, $(D_n^{+})_{n\pg1}$ respectively. We will assume that the two  filtrations \textit{commute} in the sense that for every $m,n\pg1$, we have
\begin{equation*}
E_m^{-}E_n^{+}=E_n^{+}E_m^{-}.
\end{equation*}
We start with a preliminary result.

\begin{lemm}\label{bimartingale_projection_lemma1}
Let $E(M)$ be an exact interpolation space for $(L_1(M),L_\infty(M))$ with order\-/continuous norm. 
\begin{enumerate}[nosep]
    \item If $x\in E(M)$ then $(E_n^{-}E_n^{+}(x))_{n\pg1}$ converges to $x$ in $E(M)$ for the norm topology.
    \item If $y\in E^{\times}(M)$ then $(E_n^{-}E_n^{+}(y))_{n\pg1}$ converges to $y$ in $E^{\times}(M)$ for the \text{\upshape w*}\-/topology.
\end{enumerate}
\end{lemm}
\begin{proof}
If $x\in E(M)$, then 
\begin{align*}
\|E_n^{-}E_n^{+}(x)-x\|_{E(M)}&=\|E_n^{-}(E_n^{+}(x)-x)+E_n^{-}(x)-x\|_{E(M)}\\
&\pp\|E_n^{+}(x)-x\|_{E(M)}+\|E_n^{-}(x)-x\|_{E(M)}\underset{n\to\infty}{\to}0.
\end{align*}
Now, if $y\in E^{\times}(M)$ and $x\in E(M)$ then 
\[\tau(xE_n^{-}E_n^{+}(y))=\tau(E_n^{-}E_n^{+}(x)y)\underset{n\to\infty}{\to}\tau(xy).\]
\end{proof}

Now, let $I^{-},I^{+}$ be two fixed sets of positive integers, and let $P$ be the \textit{bimartingale projection} associated with the data $(M,(M_n^{-})_{n\pg1},(M_n^{+})_{n\pg1},I^{-},I^{+})$, i.e. the projection on $L_2(M)$ such that, if $x\in L_2(M)$ then
\[P(x)=\sum_{m\in I^{-},n\in I^{+}}D_m^{-}D_n^{+}(x),\ \ \ \ \text{in}\ L_2(M).\]
In the following, we denote
\[H(P):=\big\{x\in(L_1+L_\infty)(M)\ :\ \forall m\notin I^{-},\ D_m^{-}(x)=0,\ \forall n\notin I^{+},\ D_n^{+}(x)=0\big\}\]
so that we have
\[P(L_2(M))=L_2(M)\cap H(P).\]

\begin{mdframed}[skipabove=10pt]
\begin{theo}\label{bimartingale_projection_theorem1}
Let $E(M)$ be an exact interpolation space for $(L_1(M),L_\infty(M))$ with order\-/continuous norm. Then we have
\[H_E(P)=E(M)\cap H(P),\]
\[H_{E^{\times}}(P)=E^{\times}(M)\cap H(P).\]
Moreover, the subspace
\[(L_1\cap L_\infty)(M)\cap H(P)\]
is norm-dense in $H_E(P)$ and \text{\upshape w*}-dense in $H_{E^{\times}}(P)$.
\end{theo}
\end{mdframed}
\begin{proof}
It is clear that 
\[E(M)\cap H(P):=\big\{x\in E(M)\ :\ \forall m\notin I^{-},\ D_m^{-}(x)=0,\ \forall n\notin I^{+},\ D_n^{+}(x)=0\big\}\]
is a norm-closed subspace of $E(M)$ that contains $E(M)\cap P(L_2(M))$, and moreover $E(M)\cap P(L_2(M))$ clearly contains $(L_1\cap L_\infty)(M)\cap H(P)$. Thus, it suffices to show that $(L_1\cap L_\infty)(M)\cap H(P)$ is norm-dense in $E(M)\cap H(P)$. Let $x\in E(M)\cap H(P)$. As the sequence $(E_n^{-}E_n^{+}(x))_{n\pg1}$ converges to $x$ in $E(M)$ for the norm topology and also belongs to $E(M)\cap H(P)$, we can assume that there is $n\pg1$ such that $E_n^{-}(x)=E_n^{+}(x)=x$, so that we have
\[x=\sum_{i,j=1}^{n}D_i^{-}D_j^{+}(x)=\sum_{i\in I^{-},j\in I^{+},i,j\pp n}D_i^{-}D_j^{+}(x).\]
As $(L_1\cap L_\infty)(M)$ is norm-dense in $E(M)$, there is a net $(y_\alpha)_\alpha$ of $(L_1\cap L_\infty)(M)$ that converges to $x$ in $E(M)$ for the norm topology. We set
\[x_\alpha:=\sum_{i\in I^{-},j\in I^{+},i,j\pp n}D_i^{-}D_j^{+}(y_\alpha).\]
Then $x_\alpha\in(L_1\cap L_\infty)(M)\cap H(P)$. As the net $(x_\alpha)_\alpha$ clearly converges to $x$ in $E(M)$ for the norm topology, the proof is complete. The dual statement is proved similarly.    
\end{proof}

\begin{mdframed}[skipabove=10pt]
\begin{theo}\label{bimartingale_projection_theorem2}
The bimartingale projection $P$ satisfies the technical assumptions.
\end{theo}
\end{mdframed}
\begin{proof}
Let $E(M)$ be an exact interpolation space for $(L_1(M),L_\infty(M))$ with order\-/continuous norm. It suffices to show that the polar of $H_E(P)$, $H_{E^{\times}}(P)$ w.r.t. trace duality coincides respectively with $H_{E^{\times}}(P^{\bot})$, $H_E(P^{\bot})$. Let $P_-$ and $P_+$ denote the two commuting projections on $L_2(M)$ such that, if $x\in L_2(M)$, then 
\[P_-(x)=\sum_{m\in I^{-}}D_m^{-}(x),\ \ \ \ \text{in}\ L_2(M),\]
\[P_+(x)=\sum_{n\in I^{+}}D_n^{+}(x),\ \ \ \ \text{in}\ L_2(M).\]
Then we clearly have 
\[P=P_{-}P_+=P_{+}P_-\]
and 
\[H_E(P)=H_E(P_-)\cap H_E(P_{+}).\]
As a consequence of the previous section, the polar of $H_E(P)$ w.r.t. trace duality is the w*-closure of $H_{E^{\times}}(P_{-}^{\bot})+H_{E^{\times}}(P_{+}^{\bot})$. Thus, it suffices to show that $P^{\bot}(L_2(M))\cap E^{\times}(M)$ is w*-dense in $H_{E^{\times}}(P_{-}^{\bot})+H_{E^{\times}}(P_{+}^{\bot})$. Let $x^{\pm}\in H_{E^{\times}}(P_{\pm}^{\bot})$. From the previous section we know that there is a net $(x_\alpha^{\pm})_\alpha$ of $(L_1\cap L_\infty)(M)$ that converges to $x^{\pm}$ in $E^{\times}(M)$ for the w*-topology and such that $D_n^{\pm}(x)=0$ for every $n\in I^{\pm}$. As $P_{\pm}(x_\alpha^{\pm})=0$, we deduce that $P(x_\alpha^{-}+x_\alpha^{+})=0$. Thus $x^{-}+x^{+}$ is in the w*-closure of $P^{\bot}(L_2(M))\cap E^{\times}(M)$, as desired. The dual statement is proved similarly.
\end{proof}

Now we introduce two assumptions on $(M,(M_n^{-})_{n\pg1},(M_n^{+})_{n\pg1},I^{-},I^{+})$.

\begin{mdframed}[backgroundcolor=black!10,rightline=false,leftline=false,topline=false,bottomline=false,skipabove=10pt]
$\mathbf{(\star)}$ For every $m\notin I^{-}$ and $n\notin I^{+}$, we have
\begin{equation}
D_m^{-}D_n^{+}=D_n^{+}D_m^{-}=0.
\end{equation}
\end{mdframed}
\vspace{5pt}

\begin{mdframed}[backgroundcolor=black!10,rightline=false,leftline=false,topline=false,bottomline=false,skipabove=10pt]
$\mathbf{(\star\star)}$ We have $I^{-}=\{2m-1\ :\ m\pg1\}$, $I^{+}=\{2n\ :\ n\pg1\}$, there are semifinite von Neumann algebras $A$ and $B$ equipped with filtrations $(A_n)_{n\pg1}$ and $(B_n)_{n\pg1}$ such that $M=A\bar{\otimes}B$, and such that, for $n\pg1$ we have
\begin{equation*}
\begin{aligned}
M_{2n-1}^{-}:=A_n\bar{\otimes}B_n,\ \ \ \ \ \ \ \ \ \ \ \ \ M_{2n}^{-}:=A_{n+1}\bar{\otimes}B_n,\\
M_{2n-1}^{+}:=A_n\bar{\otimes}B_{n+1},\ \ \ \ \ \ \ \ \ \ \ \ \ M_{2n}^{+}:=A_{n+1}\bar{\otimes}B_{n+1}.
\end{aligned}
\end{equation*}
\end{mdframed}
\vspace{5pt}

\begin{lemm}\label{bimartingale_projection_lemma4}
$\mathbf{(\star\star)}$ is stronger than $\mathbf{(\star)}$.
\end{lemm}
\begin{proof}
Assume $\mathbf{(\star\star)}$ holds. Let $(E_n^A)_{n\pg1}$ and $(E_n^B)_{n\pg1}$ respectively denote the the conditional expectations associated with the filtrations $(A_n)_{n\pg1}$ and $(B_n)_{n\pg1}$. Then, if $m\pg1$ we have
\[D^{-}_{2m}D^{+}_1=(E^A_{m+1}\otimes E^B_m-E^A_m\otimes E^B_m)(E^A_1\otimes E_2^B)=(E^A_{m+1}-E^A_m)E_1^A\otimes E_m^BE_1^B=0.\]
If $m\pg1$ and $n\pg1$, we have
\begin{align*}
    D^{-}_{2m}D^{+}_{2n-1}&=(E^A_{m+1}\otimes E^B_m-E^A_m\otimes E^B_m)(E^A_n\otimes E^B_{n+1}-E^A_n\otimes E^B_n)\\
    &=\big((E^A_{m+1}-E^A_m)\otimes E^B_m\big)\big((E^A_n\otimes(E^B_{n+1}-E^B_n)\big)\\
    &=E_n^A(E^A_{m+1}-E^A_m)\otimes E^B_m(E^B_{n+1}-E^B_n).
\end{align*}
But if $m\pp n$, then $E^B_m(E^B_{n+1}-E^B_n)=0$, and if $m\pg n$, then $E_n^A(E^A_{m+1}-E^A_m)=0$. Thus, $D^{-}_{2m}D^{+}_{2n-1}$ is always $0$, as required.
\end{proof}

\begin{mdframed}[skipabove=10pt]
\begin{theo}\label{bimartingale_projection_theorem7}
Assume $\mathbf{(\star\star)}$ holds. Then the bimartingale projection $P$ is of Jones-type.
\end{theo}
\end{mdframed}
\begin{proof}
As the composition of two martingale projections, we know that $P$ is $L_p$-bounded with a constant depending only on $p$, for every $1<p<\infty$. It is proved in \cite{Moyart2024}[Lemma 2.13] that, under $\mathbf{(\star)}$, the subcouple $(H_1(P^{\bot}),H_2(P^{\bot}))$ is $K$-closed in $(L_1(M),L_2(M))$ with a universal constant. Finally, it is proved in \cite{Moyart2024}[Fact 14] that, under $\mathbf{(\star\star)}$, the subcouple $(H_1(P),H_2(P))$ is $K$-closed in $(L_1(M),L_2(M))$ with a universal constant. The proof is complete.
\end{proof}

The previous results show that the projection $P$ satisfies the technical conditions. 

\begin{lemm}\label{bimartingale_projection_lemma7}
Let $N$ be a tracial von Neumann algebra. Then the amplified projection $Q:=P\otimes I$ on $L_2(M\bar{\otimes}N)=L_2(M)\otimes L_2(N)$ is the bimartingale projection associated with the data 
\[(M\bar{\otimes}N,(M_n^{-}\bar{\otimes}N)_{n\pg1},(M_n^{+}\bar{\otimes}N)_{n\pg1},I^{-},I^{+}).\]
\end{lemm}

\begin{mdframed}[skipabove=10pt]
\begin{theo}\label{bimartingale_projection_theorem8}
Assume $\mathbf{(\star\star)}$ holds. Then the bimartingale projection $P$ satisfies Pisier's method assumptions.
\end{theo}
\end{mdframed}
\begin{proof}
Let $N$ be a semifinite von Neumann algebra. As the data $(M,(M_n^{-})_{n\pg1},(M_n^{+})_{n\pg1},I^{-},I^{+})$ satisfies $\mathbf{(\star)}$, we see that this is also the case for the data
\[(M\bar{\otimes}N,(M_n^{-}\bar{\otimes}N)_{n\pg1},(M_n^{+}\bar{\otimes}N)_{n\pg1},I^{-},I^{+}).\]
Thus, by Lemma \ref{bimartingale_projection_lemma7} together with Theorem \ref{bimartingale_projection_theorem7}, we deduce that the amplified projection $Q:=P\otimes I$ on $L_2(M\bar{\otimes}N)$ is of Jones-type. Thus, it only remains to justify that $H_\infty(P^{\bot})\odot L_\infty(N)$ is included in $H_\infty(P^{\bot}\otimes I)$. Let $x\in H_\infty(P^{\bot})$ and $f\in L_\infty(N)$. As $H_\infty(P^{\bot}\otimes I)$ is the polar of $H_1(P\otimes I)$ w.r.t. trace duality, it suffices to show that $\tau((x\otimes f)w)=0$ for every $w\in H_1(P\otimes I)$. Let $w\in H_1(P\otimes I)$. As the sequence $((E_n^{-}\otimes I)(E_n^{+}\otimes I)(w))_{n\pg1}$ converges to $w$ in $L_1(M\bar{\otimes}N)$, and also belongs to $H_1(P\otimes I)$, we can assume that there is $n\pg1$ such that $(E_n^{-}\otimes I)(w)=(E_n^{+}\otimes I)(w)=w$, so that we have
\[w=\sum_{i\in I^{-},j\in I^{+},i,j\pp n}(D_i\otimes I)^{-}(D_j^{+}\otimes I)(w).\]
Besides, as the algebraic tensor product $L_1(M)\odot L_1(N)$ is norm-dense in $L_1(M\bar{\otimes}N)$, there is a net $(w_\alpha)_\alpha$ of $L_1(M)\odot L_1(N)$ that converges to $w$ in $L_1(M\bar{\otimes}N)$. Finally, we set
\[w_\alpha':=\sum_{i\in I^{-},j\in I^{+},i,j\pp n}(D_i\otimes I)^{-}(D_j^{+}\otimes I)(w_\alpha).\]
Then clearly the net $(w_\alpha')_\alpha$ converges to $w$ in $L_1(M\bar{\otimes}N)$. Thus, it suffices to check that $(\tau\otimes\sigma)((x\otimes f)w_\alpha')=0$. But we can write
\[w_\alpha=\sum_\beta y_{\alpha\beta}\otimes g_{\alpha\beta}\]
with $y_{\alpha\beta}\in L_1(M)$ and $g_{\alpha\beta}\in L_1(N)$. Thus, we have
\begin{align*}
w_\alpha'&=\sum_\beta\sum_{i\in I^{-},j\in I^{+},i,j\pp n}(D_i\otimes I)^{-}(D_j^{+}\otimes I)(y_{\alpha\beta}\otimes g_{\alpha\beta})\\
&=\sum_\beta\sum_{i\in I^{-},j\in I^{+},i,j\pp n}D_i^{-}D_j^{+}(y_{\alpha\beta})\otimes g_{\alpha\beta}\\
&=\sum_\beta\Big[\sum_{i\in I^{-},j\in I^{+},i,j\pp n}D_i^{-}D_j^{+}(y_{\alpha\beta})\Big]\otimes g_{\alpha\beta}\\
&=\sum_\beta y_{\alpha\beta}'\otimes g_{\alpha\beta}
\end{align*}
where 
\[y_{\alpha\beta}':=\sum_{i\in I^{-},j\in I^{+},i,j\pp n}D_i^{-}D_j^{+}(y_{\alpha\beta})\]
clearly belongs to $H_1(P)$. Thus, we have $\tau(xy_{\alpha\beta}')=0$ and then
\begin{align*}
(\tau\otimes\sigma)((x\otimes f)w_\alpha')&=\sum_\beta(\tau\otimes\sigma)((x\otimes f)(y_{\alpha\beta}'\otimes g_{\alpha\beta}))\\
&=\sum_\beta(\tau\otimes\sigma)(xy'_{\alpha\beta}\otimes fg_{\alpha\beta})\\
&=\sum_\beta\tau(xy'_{\alpha\beta})\sigma(fg_{\alpha\beta})=0
\end{align*}
as desired.
\end{proof}

\begin{mdframed}[skipabove=10pt]
\begin{coro}\label{bimartingale_projection_corollary1}
Assume $\mathbf{(\star\star)}$ holds. If $1<p<\infty$ and $0<\theta<1$, then
\[[H_p(P),H_1(P)]_\theta=H_{p_\theta}(P)\]
with equivalent norms, with constants depending only on $p,\theta$, where 
\[\frac{1}{p_\theta}=\frac{1-\theta}{p}+\theta.\]
\end{coro}
\end{mdframed}
\vspace{5pt}

\section{Martingale Hardy spaces}

Let $M$ be a semifinite von Neumann algebra equipped with a filtration $(M_n)_{n\pg1}$ with associated conditional expectations denoted $(E_n)_{n\pg1}$ and associated increment projections denoted $(D_n)_{n\pg1}$. For $1\pp p\pp\infty$, and $\dagger\in\{r,c,rc\}$, let $L_p(M,\ell_2^\dagger)$ denote the associated row/column/mixed $\ell_2$-valued $L_p$-space, as introduced in the seminal paper \cite{PisierXuMartingales}. 

A sequence $(x_n)_{n\pg1}$ of $(L_1+L_\infty)(M)$ is \textit{adapted} if $E_n(x_n)=x_n$ for every $n\pg1$. For $1\pp p\pp\infty$, and $\dagger\in\{r,c,rc\}$, we set 
\[L_p^{\text{\upshape ad}}(M,\ell_2^\dagger):=\big\{x=(x_n)_{n\pg1}\in L_p(M,\ell_2^\dagger)\ :\ (x_n)_{n\pg1}\ \text{\ is\ an\ adapted\ sequence}\big\}.\]
It is clear that they are closed subspaces of $L_p(M,\ell_2^\dagger)$. Using the transference techniques mentionned in \cite{Moyart2024}[Lemma 3.1], the following result is a direct consequence of Corollary \ref{martingale_projection_corollary1}.

\begin{mdframed}[skipabove=10pt]
\begin{theo}\label{Musat_theorem_theorem1}
Let $\dagger\in\{r,c,rc\}$. If $1<p<\infty$ and $0<\theta<1$, then
\[[L_p^{\text{\upshape ad}}(M,\ell_2^\dagger),L_1^{\text{\upshape ad}}(M,\ell_2^\dagger)]_{\theta}=L_{p_\theta}^{\text{\upshape ad}}(M,\ell_2^\dagger)\]
with equivalent norms, with constants depending only on $p,\theta$, where 
\[\frac{1}{p_\theta}=\frac{1-\theta}{p}+\theta.\]
\end{theo}
\end{mdframed}

\begin{rem}
In the case $\dagger\in\{r,c\}$, the above theorem was first established in \cite{Narcisse2021}[Proposition 3.18], where the result is in fact proved in the broader quasi-normed setting. The case $\dagger=rc$ appears to be new.
\end{rem}

A sequence $(x_n)_{n\pg1}$ of $(L_1+L_\infty)(M)$ is a  \textit{martingale increment} if it is adapted and if $E_{n-1}(x_n)=0$ for every $n\pg2$. For $1\pp p\pp\infty$, and $\dagger\in\{r,c,rc\}$, we set 
\[L_p^{\text{\upshape mi}}(M,\ell_2^\dagger):=\big\{x=(x_n)_{n\pg1}\in L_p(M,\ell_2^\dagger)\ :\ (x_n)_{n\pg1}\ \text{\ is\ a\ martingale\ increment}\big\}.\]
It is clear that they are closed subspaces of $L_p^{\text{ad}}(M,\ell_2^\dagger)$. By using the same transference techniques together with Corollary \ref{bimartingale_projection_corollary1}, we obtain the following result.

\begin{mdframed}[skipabove=10pt]
\begin{theo}\label{Musat_theorem_theorem2}
Let $\dagger\in\{r,c,rc\}$. If $1<p<\infty$ and $0<\theta<1$, then
\[[L_p^{\text{\upshape mi}}(M,\ell_2^\dagger),L_1^{\text{\upshape mi}}(M,\ell_2^\dagger)]_{\theta}=L_{p_\theta}^{\text{\upshape mi}}(M,\ell_2^\dagger)\]
with equivalent norms, with constants depending only on $p,\theta$, where 
\[\frac{1}{p_\theta}=\frac{1-\theta}{p}+\theta.\]
\end{theo}
\end{mdframed}

\begin{rem}
The case $\dagger\in\{r,c\}$ is equivalent to \cite{MusatHardy}[Theorem 3.1]. It can also be direclty deduced from the case $\dagger\in\{r,c\}$ in Theorem \ref{Musat_theorem_theorem1} together with Stein's inequality and Lépingle-Yor inequality for noncommutative martingales respectively proved in \cite{PisierXuMartingales}[Theorem 2.3] and \cite{QiuLepingle}[Theorem 2.1].
\end{rem}

For $1\pp p<\infty$, let $\mathscr{H}_p(M)$ denote the associated mixed martingale Hardy space, defined by
\[\mathscr{H}_p(M)=\left\{\begin{array}{cl}
    L_p^{\text{\upshape mi}}(M,\ell_2^r)+L_p^{\text{\upshape mi}}(M,\ell_2^c)&\text{if}\ 1\pp p<2\\
    L_p^{\text{\upshape mi}}(M,\ell_2^r)\cap L_p^{\text{\upshape mi}}(M,\ell_2^c)&\text{if}\ p\pg2
\end{array}\right.\]
As detailed in \cite{MusatHardy}, the following result is a direct consequence of the case $\dagger\in\{r,c\}$ in Theorem \ref{Musat_theorem_theorem2} together with the noncommutative Burkholder-Gundy inequality proved in \cite{JungeXuBurkholder}.

\begin{mdframed}[skipabove=10pt]
\begin{theo}[Musat]\label{Musat_theorem_theorem3}
Assume that $M$ is finite. If $1<p<\infty$ and $0<\theta<1$, then
\[[\mathscr{H}_p(M),\mathscr{H}_1(M)]_{\theta}=\mathscr{H}_{p_\theta}(M)\]
with equivalent norms, with constants depending only on $p,\theta$, where 
\[\frac{1}{p_\theta}=\frac{1-\theta}{p}+\theta.\]
\end{theo}
\end{mdframed}

\begin{rem}
It is clear that
\[\mathscr{H}_p(M)=L_p^{\text{\upshape mi}}(M,\ell_2^{rc}),\ \ \ \ \ \text{for}\ p\pg2.\]
Actually, as a consequence of Stein inequalities for noncommutative martingales, we have
\[\mathscr{H}_p(M)=L_p^{\text{\upshape mi}}(M,\ell_2^{rc}),\ \ \ \ \ \text{for}\ p>1.\]
with equivalent norms, with constants depending only on $p$. However, there is no reason for the two spaces $L_1^{\text{\upshape mi}}(M,\ell_2^{rc})$ and $\mathscr{H}_1(M)$ to be equal. Nevertheless, if $x\in\mathscr{H}_1(M)$, then clearly $x\in L_1^{\text{\upshape mi}}(M,\ell_2^{rc})$ with 
\[\|x\|_{L_1^{\text{\upshape mi}}(M,\ell_2^{rc})}\pp\|x\|_{\mathscr{H}_1(M)}.\] 
\end{rem}

\clearpage

\part*{Acknowledgements and References}
\addcontentsline{toc}{section}{Acknowledgements}

\begin{center}
    Acknowledgements
\end{center}

\vspace{4pt}

I am very grateful to my advisor Éric Ricard for many valuable discussions and his guidance throughout the writing of this article.

\bibliographystyle{plain}
\bibliography{BIBLIOGRAPHY} 

\end{document}